\documentclass[11pt]{amsart}
\usepackage{geometry}                
\usepackage{graphicx}
\usepackage{amssymb}
\usepackage{epstopdf}
\newtheorem{thm}{Theorem}[section]
\newtheorem{lem}{Lemma}[section]
\newtheorem{cor}{Corollary}[section]

\newtheorem{conj}{Conjecture}[section]
\newtheorem{defn}{Definition}[section]

\numberwithin{equation}{section}

\def\Z{\mathbb Z}

\def\R{\mathbb R}

\def\d{\partial}
\def\a{\alpha}
\def\b{\beta}
\def\g{\gamma}

\def\e{\epsilon}

\title{Stratified convexity \& concavity of Gradient Flows on Manifolds with Boundary}
\author{Gabriel Katz}

\address{5 Bridle Path Circle, Framingham, MA 01701, USA}
\email{gabkatz@gmail.com}

\begin{document}
\maketitle

\begin{abstract} As has been observed by Morse \cite{Mo}, any generic vector field $v$ on a compact smooth manifold $X$ with boundary gives rise to a stratification of the boundary $\d X$ by compact submanifolds $\{\d_j^\pm X(v)\}_{1 \leq j \leq \dim(X)}$, where  $\textup{codim}(\d_j^\pm X(v))= j$. 

Our main observation is that this stratification reflects  the stratified convexity/concavity of the boundary $\d X$ with respect to the $v$-flow. We study the behavior of this stratification under deformations of the vector field $v$.  We also investigate the restrictions that the existence of a convex/concave traversing $v$-flow imposes on the topology of $X$. 

Let $v_1$ be the orthogonal projection of $v$ on the tangent bundle of $\d X$. We link the dynamics of the $v_1$-flow on the boundary with the property of $v$ in $X$ being convex/concave. This linkage is an instance of more general phenomenon that we call ``holography of traversing fields"---a subject of a different paper to follow.
\end{abstract}

\section{Introduction}

This paper is the first in a series that investigates the Morse Theory and gradient flows on smooth compact manifolds with boundary, a special case of the well-developed Morse theory on stratified spaces (see \cite{GM}, \cite{GM1}, and \cite{GM2}). For us, however, the starting starting point and the source of inspiration is the 1929 paper of Morse \cite{Mo}. 

We intend to present to the reader a version of the Morse Theory in which the critical points remain behind the scene, while shaping the geometry of the boundary!  Some of the concepts that animate our approach can be found in \cite{K}, where they are adopted to the special environment a $3D$-gradient flows. These notions include \emph{stratified convexity or concavity} of traversing flows in connection to the boundary of the manifold. That concavity serves as a measure of \emph{intrinsic complexity} of a given manifold $X$ with respect to any traversing flow.  Both convexity and concavity have strong topological implications. 

Another central theme that will make its first brief appearance in this paper is \emph{the holographic properties of traversing flows} on manifolds with boundary. The ultimate aim here is to reconstruct (perhaps, only partially) the bulk of the manifold and the dynamics of the flow on it from some residual structures on the boundary. Thus the name ``holography".
\smallskip

In Section 2, for so-called \emph{boundary generic} fields $v$ on $X$ (see Definition \ref{def2.1}), we explore the Morse stratification $\{\d_j^\pm X(v)\}_j$ of the boundary $\d X$ (see formula \ref{eq2.1} and \cite{Mo}, induced by the vector field $v$ on $X$. 

In Section 3, we investigate the degrees of freedom to change this stratification by deforming a given vector field within the space of gradient-like fields (Theorem \ref{th3.2}, Corollary \ref{cor3.2}, and Corollary \ref{cor3.3}). 

In Section 4,  for vector fields on compact manifolds, we introduce the pivotal notion of boundary $s$-convexity/$s$-concavity, $s = 2, 3, \dots $ (see Definition \ref{def4.1}). Then we explore some topological implications of the existence of a boundary $2$-convex/$2$-concave traversing field on $X$ (see Lemma \ref{lem4.2}, Corollary \ref{cor4.2}, Corollary \ref{cor4.3}, and Corollary  \ref{cor4.4}).   

Let $v_1$ denote the orthogonal projection of the field $v|_{\d X}$ on the bundle $T(\d X)$ tangent to the boundary. Occasionally, we can determine whether a given field $v$ is convex/concave just by observing the behavior of the $v_1$-trajectories on the boundary $\d_1X$ (Theorem \ref{th4.1}, Theorem \ref{th4.2}). We view the possibility of such determination as an instance of a more general phenomenon, which we call ``holography". This phenomenon will occupy us fully in a different paper.

The Eliashberg surgery theory of folding maps \cite{E1}, \cite{E2} helps us to describe the patterns of Morse stratifications for traversing $3$-concave and $3$-convex fields (Theorem \ref{th5.1}, Conjecture \ref{conj5.1},  and Corollary \ref{cor5.1}).

\section{The Morse  Stratification $\{\d_j^+X(v)\}$}
Inspired by \cite{Mo}, we start by introducing some basic notions and constructions that describe the way in which generic vector fields on a compact smooth manifold interact with its boundary.

Let $X$ be a compact smooth $(n + 1)$-dimensional manifold with a boundary $\d X$. Let $v$ be a smooth vector field on $X$ which does not vanish on the boundary $\d X$.  As a rule, we assume that $X$ is properly contained in a $(n + 1)$-dimensional manifold $\hat X$ and that the field $v$ extends to a field $\hat v$ on $\hat X$ so that $v|_{\hat X \setminus X} \neq 0$. In fact, we always treat the pair $(\hat X, \hat v)$ as a \emph{germ} of a space and a field in the vicinity of the given pair $(X, v)$.

\begin{figure}[ht]
\centerline{\includegraphics[height=1.5in,width=2in]{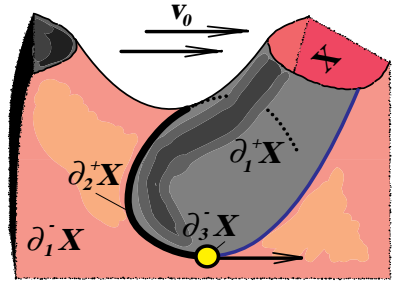}} 
\bigskip
\caption{The Morse stratification generated by the horizontal field $v_0$ on a solid $X$ bounded by the saddle surface $\d_1X$.}
\end{figure}

\begin{figure}[ht]
\centerline{\includegraphics[height=1.8in,width=4in]{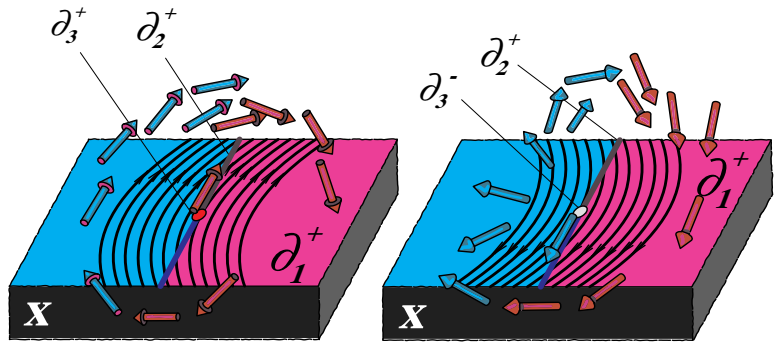}} 
\bigskip
\caption{A generic field $v$ in the vicinity of a cusp point on the boundary of a solid $X$ generates the Morse stratification $\d_3^+X \subset \d_2^+X \subset \d_1^+X$ (the left diagram) or  the Morse stratification $\d_3^-X \subset \d_2^+X \subset \d_1^+X$ (the right diagram).}
\end{figure}

Often we will consider vector fields only with the isolated \emph{Morse-type} singularities (zeros) located away from the boundary.  This means that $v$, viewed as a section of the tangent bundle $T(X)$, is transversal  its zero section. In other words, in the vicinity of each singular point, there is a local system of coordinates $(x_1, \dots , x_{n + 1})$ such that the field $v$ can be represented as $v = (a_1x_1,\, \dots , \, a_{n + 1}x_{n+1})$, where all $a_i \neq 0$.

To achieve some uniformity in our notations, let $\d_0X := X$ and $\d_1X := \d X$.

The vector field $v$ gives rise to a partition $\d_1^+X \cup \d_1^-X $ of the boundary $\d_1X$ into  two sets: the locus $\d_1^+X$, where the field is directed inward of $X$, and  $\d_1^-X$, where it is directed outwards. We assume that $v$, viewed as a section of the quotient  line bundle $T(X)/T(\d X)$ over $\d X$, is transversal to its zero section. This assumption implies that both sets $\d^\pm_1 X$ are compact manifolds which share a common boundary $\d_2X := \d(\d_1^+X) = \d(\d_1^-X)$. Evidently, $\d_2X$ is the locus where $v$ is \emph{tangent} to the boundary $\d_1X$.

Morse has noticed that, for a generic vector field $v$, the tangent locus $\d_2X$ inherits a similar structure in connection to $\d_1^+X$, as $\d _1X$ has in connection to $X$ (see \cite{Mo}). That is, $v$ gives rise to a partition $\d_2^+X \cup \d_2^-X $ of  $\d_2X $ into  two sets: the locus $\d_2^+X$, where the field is directed inward of $\d_1^+X$, and  $\d_2^-X$, where it is directed outward of $\d_1^+X$. Again, let us assume that $v$, viewed as a section of the quotient  line bundle $T(\d_1X)/T(\d_2X)$ over $\d_2X$, is transversal to its zero section.

For generic fields, this structure replicates itself: the cuspidal locus $\d_3X$ is defined as the locus where $v$ is tangent to $\d_2X$; $\d_3X$ is divided into two manifolds, $\d_3^+X$ and $\d_3^-X$. In  $\d_3^+X$, the field is directed inward of $\d_2^+X$, in  $\d_3^-X$, outward of $\d_2^+X$. We can repeat this construction until we reach the zero-dimensional stratum $\d_{n+1}X = \d_{n+1}^+X \cup  \d_{n+1}^-X$. 

These considerations motivate 

\begin{defn}\label{def2.1} 
We say that a smooth field $v$ on $X$ is \emph{boundary generic} if:
\begin{itemize}
\item $v|_{\d X} \neq 0$,
\item $v$, viewed as a section of the tangent bundle $T(X)$, is transversal to its zero section,
\item  for each $j = 1, \dots,  n+1$, the $v$-generated stratum $\d_jX$ is a  smooth submanifold of  $\d_{j-1}X$,
\item  the field  $v$, viewed as section of the quotient 1-bundle  $$T_j^\nu := T(\d_{j-1}X)/ T(\d_jX) \to \d_jX,$$ is transversal to the zero section of $T_j^\nu$ for all $j > 0$. 
\end{itemize}

We denote the space of smooth generic vector fields on $X$ by the symbol $\mathcal V^\dagger(X)$. \hfill\qed
\end{defn}

Thus a boundary generic vector field $v$  on $X$  gives rise to two  stratifications: 
\begin{eqnarray}\label{eq2.1}
\d X := \d_1X \supset \d_2X \supset \dots \supset \d_{n +1}X, \nonumber \\ 
X := \d_0^+ X \supset \d_1^+X \supset \d_2^+X \supset \dots \supset \d_{n +1}^+X 
\end{eqnarray}
, the first one by closed submanifolds, the second one---by compact ones.  Here $\dim(\d_jX) = \dim(\d_j^+X) = n +1 - j$. For simplicity, the notations ``$\d_j^\pm X$" do not reflect the dependence of these strata on the vector field $v$. When the field varies, we use a more accurate notation ``$\d_j^\pm X(v)$".

\smallskip
\noindent {\bf Remark 2.1.} Replacing $v$ with $-v$ affects the Morse stratification according to the formula: $$\d_j^+X(-v) = \d_j^\e X(v)$$, where $\e = +$ when $(n + 1) - j \equiv 0 \;\mod  (2)$, and $\e = -$ otherwise. \hfill \qed
\smallskip

We will postpone the proof of the theorem below until the second paper in this series of articles (see \cite{K3}, Theorem 6.6, an extension of Theorem \ref{th2.1}  below). There we will develop the needed analytical tools.

\begin{thm}\label{th2.1} Boundary generic  vector fields  form an open and dense subset $\mathcal V^\dagger(X)$ in the space $\mathcal V(X)$ of all smooth fields on $X$. 
\hfill \qed
\end{thm} 

\begin{defn}\label{def2.2} We say  that a smooth vector field $v$ on $X$ is \emph{of the gradient type (or gradient-like)} for a  smooth function $f: X \to \R$ if: 
\begin{itemize}
\item the differential $df$ and the field $v$ vanish on the same locus $Z \subset X$, 
\item the function $df(v) > 0$ in $X \setminus Z$,
\item in the vicinity of $Z$, there exist a Rimannian metric $g$ on $X$ so that $v = \nabla_gf$, the gradient field of $f$ in the metric $g$. \hfill \qed
\end{itemize}
\end{defn}

\begin{defn}\label{def2.3} A smooth function $f: X \to \R$ is called \emph{Morse function} if its differential $df$, viewed as a section of the cotangent bundle $T^\ast(X)$, is transversal to the zero section. \hfill \qed
\end{defn}

Recall that, for a Morse function $f$ on a compact  $(n +1)$-manifold $X$,  the critical set $Z :=\{ x \in X |\;  df_x = 0\}$  is finite and each point $x \in Z$ has special local coordinates $(x_1, \dots , x_{n+1})$ such that $df = \sum_{1 \leq i \leq n+1} a_ix_i dx_i$, where $a_i \neq 0$ for all $i$ (for example, see \cite{GG}).

\begin{defn}\label{def2.4} Let $f: X \to \R$ be a smooth function and $v$ its gradient-like vector field. We say that the pair $(f, v)$ is \emph{boundary generic} if the field $v$ is boundary generic in the sense of Definition \ref{def2.1} and the restrictions of $f$ to each stratum $\d_jX := \d_jX(v)$ are Morse functions for all $0 \leq j \leq n$. \hfill \qed
\end{defn}

\begin{lem}\label{lem2.1} Let $V$ be a compact smooth manifold, and $Y$ a smooth manifold which is stratified by submanifolds $\{Y_j\}_j$. Let  $\mathcal Z = \mathcal Z(V, Y)$ be the space of smooth maps $\Psi: V \to Y$ which are transversal  to each stratum $Y_j$. Put $V_j^\Psi := \Psi^{-1}(Y_j)$. Next consider the space $\mathcal X = \mathcal X(V, Y)$ of pairs $(f, \Psi)$ such that $\Psi \in  \mathcal Z$ and $f: V \to \R$ has the property: $\{f|_{V_j^\Psi}\}_j$ are Morse functions for all $j$. Then $\mathcal X$ is open and dense in the space $C^\infty(V, Y\times \R)$.
\end{lem}

\begin{proof} Consider the space $(T^\ast V) \times Y$, where $T^\ast V$ denotes the cotangent bundle of $V$. The property $(f, \Psi) \in \mathcal X$ is equivalent to the property of the section $df$ of the bundle $$T^\ast V \times Y \to V \times Y$$  to be transversal to each (transversal) intersection of the $\Psi$-graph $\Gamma_\Psi \subset V \times Y$ with each stratum $V \times Y_j$.  The latter property defines a open set in $C^\infty(V, Y\times \R)$.  

In order to validate density of $\mathcal X$ in $C^\infty(V, Y\times \R)$, we first perturb a given map $\Psi: V \to Y$ to make it transversal to each stratum $Y_j \subset Y$, and then perturb a given function $f: V \to \R$ to make the section $df$ of $T^\ast V$ transversal to each manifold $V_j^\Psi := \Psi^{-1}(Y_j)$. 
\end{proof}

\begin{thm}\label{th2.2} The  boundary generic\footnote{in the sense of Definition \ref{def2.4}} Morse pairs $(f, v)$ on a compact manifold $X$ form an open and dense subset  in the space of all smooth functions $f: X \to \R$ and their gradient-like fields $v$. 
\end{thm}

\begin{proof}  By Theorem \ref{th2.1}, the boundary generic fields $v$ form an open and dense set in the space of all fields.
\smallskip
 
Let $\mathsf F^n$ be a complete flag in $\R^n$, formed by subspaces $F_j$ of codimension $j$. In the proof of Theorem 3.4 \cite{K3}, for every field $v$,  we will construct a smooth map $\Psi^\d(v): \d_1X \to \R^n$ such that $\d_jX(v)  = \Psi^\d(v)^{-1}(F_j)$. Moreover, $\Psi^\d(v)$ is transversal to each $F_j$, if and only if,  $v$ is a boundary generic field. The construction of the map $\Psi^\d(v)$ utilizes high order Lie derivatives $\{\mathcal L_v^j\}_{0 \leq j \leq n}$ of an auxiliary function $z: X \to \R$ as in Lemma 3.1 \cite{K3}.

Now the property of boundary generic Morse pairs $(f, v)$ to be open and dense in the space of all pairs follows from Lemma \ref{lem2.1}: just let $V = \d_1X$, $Y = \R^n$, $Y_j = F_j$, and $\Psi = \Psi^\d(v)$ in that lemma. 
\bigskip

For the reader convenience, let us sketch now an alternative argument that establishes just the density of boundary generic Morse pairs $(f, v)$ in the space of all pairs.  It does not rely on the construction of the map $\Psi^\d(v)$ from \cite{K3}. 

We start with  a pair $(f, v)$ where $v|_{\d X} \neq 0$ and $df(v) > 0$ at the points of the set where $v \neq 0$.  By a small perturbation of $f$, we can assume the $f$ is a Morse function on $X$ and $v$ its gradient-like field. 

Let $K \supset \d X$ be a compact regular neighborhood of $\d X$ in $X$ so small that $v_K \neq 0$.  By Theorem \ref{th2.1}, we can perturb  $v$  to a new field $\tilde v$ so that  $\tilde v$ is boundary generic in the sense of Definition \ref{def2.1} and still $\tilde v|_{K} \neq 0$. 

For a given $f$, the condition $df(u)|_K > 0$ defines an open cone in the space of all  fields $u$, subject to the constraint $u|_K \neq 0$.  Therefore $\tilde v$ can be chosen both boundary generic and gradient-like for $f|_K$. When $\tilde v|_K$ is fixed, so are the stratifications $\{\d_j^+X(\tilde v) \subset \d_jX(\tilde v)\}_j$.  

Next, with $\tilde v|_K$ being fixed, we perturb $f$ again to a new function $\tilde f$ so that $d\tilde f(\tilde v)|_K > 0$ and $\{f|_{\d_jX(\tilde v)}\}$ are Morse function for all $j$. The perturbation will be supported in the compact $K$. We start constructing $\tilde f$ inductively  first from adjusting it on the 1-manifold $\d_nX(\tilde v)$ and then moving sequentially to the strata $\d_jX$ with lower indices $j$. We pick  each perturbation $\tilde f$ so small that the open condition  $d\tilde f(\tilde v)|_K > 0$ is not violated. The existence of the desired $j$-th perturbation is based on the fact that Morse functions on a compact manifold $Y$ (in this case, on $\d_jX(\tilde v)$) form an  open and dense subset in $C^\infty(Y)$, the space of all smooth functions on $Y$, being equipped with the Whitney topology. Note that since $\tilde v$ is tangent to $\d_jX(\tilde v)$ along $\d_{j +1}X(\tilde v)$ and  $d\tilde f(\tilde v)|_{\d_{j +1}X(\tilde v)} > 0$, the restriction $\tilde f|_{\d_jX(\tilde v)}$ has no critical points in the vicinity of $\d_{j +1}X(\tilde v)$. Thus we need to perturb $\tilde f|_{\d_jX(\tilde v)}$ only on a compact subset $Q_j \subset \d_jX(\tilde v)$ which has an empty intersection with $\d_{j +1}X(\tilde v)$. This perturbation extends smoothly from $Q_j$ to $X$. Eventually, we reach the upper stratum $\d_0X := X$, thus constructing a boundary generic approximation of the given pair $(f, v)$. 

All the changes $(\tilde f, \tilde v)$ of $(f, v)$, but the first one, we have introduced so far are supported in $K$, where $\tilde v \neq 0$ and $d\tilde f(\tilde v) > 0$. This proves that the boundary generic pairs form a \emph{dense} set in the space of all pairs $(f, v)$, where $v$ being a $f$-gradient-like field, subject to the constraints: $v|_{\d X} \neq 0$, and $f: X \to \R$ being a Morse function.
\end{proof}

\begin{figure}[ht]
\centerline{\includegraphics[height=1.8in,width=3in]{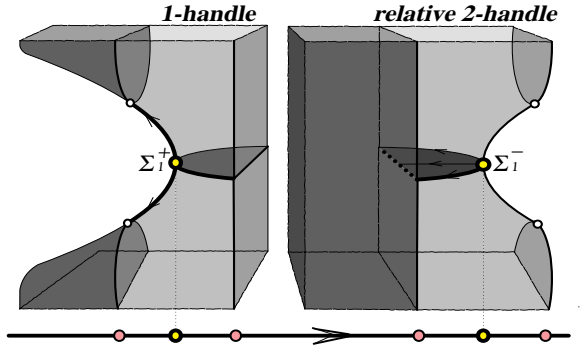}} 
\bigskip
\caption{Positive (the left diagram) and negative (the right diagram) singularities on the boundary of a solid.}
\end{figure}

For a given Morse pair $(f, v)$, we denote by $\Sigma_j \subset \d_jX(v)$ the set of critical points  of the function $f|: \d_jX \to \R$. For a boundary generic Morse pair $(f, v)$, the finite critical set $\Sigma_j$ is divided into two complementary sets: the set $\Sigma_j^+ \subset \d_j^+X$ of positive critical points and the set $\Sigma_j^- \subset \d_j^-X$ of negative ones (see Fig. 3). \smallskip

\noindent {\bf Remark 2.2.} Note that when $\d_j^+X \neq \emptyset$, it may happen that  $\Sigma_j^+  = \emptyset$. However, if  a component $\d_j^+X_\a$ of $\d_j^+X$ 
is a \emph{closed} manifold, then $f: \d_j^+X_\a \to \R$ must have local extrema, in which case $\Sigma_j^+ \neq \emptyset$. \hfill\qed
\smallskip

Consider a generic  field $v$ and a Riemannian metric $g$ on $X$. We denote by $v_j$ the orthogonal projection of the field $v$ on the tangent space $T(\d_jX)$.  
Note that if $v$ is a gradient field  for a function $f: X \to \R$ in metric $g$, then $v_j$ is automatically a gradient field for the restrictions $f|_{\d_jX}$ and $g|_{\d_jX}$.   
\smallskip

Take a smooth vector field $v$ on a compact $(m+1)$-manifold $Y$ with isolated singularities $\{y_\star \in \Sigma(v) \subset \mathsf{int}(Y)\}$.  We denote 
by $\mathsf{ind}_{y_\star}(v)$ the localized index of $v$ at its typical singular point $y_\star$. In a local chart, $\mathsf{ind}_{y_\star}(v)$ is defined as the degree of a map $G_v: S^m_{y_\star} \to S^m$ from a small $y_\star$-centered $m$-sphere  to the unit $m$-sphere. The map takes each point $a \in S_{y_\star}$ to the point $v(a)/\|v(a)\| \in S^m$.

We define the ``global" index $\mathsf{Ind}(v)$ as the sum $\sum_{y_\star \in \Sigma(v)} \mathsf{ind}_{y_\star}(v)$. 

For a generic field $v$ and a Riemannian metric $g$ on $X$, we form the fields $\{v_j\}$  on $\{\d_jX(v)\}$ and define the global index of $v_j$ by the formula:  $$\mathsf{Ind}^+(v_j) := \sum_{\{x_\star \in \Sigma_j^+\}} \; \mathsf{ind}_{x_\star}(v_j).$$ 
\smallskip

Let us revisit  the beautiful Morse formulas \cite{Mo}:

\begin{thm}[{\bf The Morse Law of Vector Fields}]\label{th2.3} \hfill\break  
For a boundary generic vector field $v$ and a Riemannian metric on a $(n +1)$-manifold $X$, such that the singularities of the fields $v_j$ are isolated for all $j \in [0, n+1]$,  the following two equivalent sets of formulas hold:
\begin{eqnarray}
\chi(\d_j^+X) =  \mathsf{Ind}^+(v_j) + \mathsf{Ind}^+(v_{j+1}) \nonumber
\end{eqnarray} 

\begin{eqnarray}\label{eq2.2}
\mathsf{Ind}^+(v_j)  =  \sum_{k = j}^{n +1} \; (-1)^k \chi(\d_k^+X)
\end{eqnarray} 
, where $\chi(\sim)$ stands for the Euler number of the appropriate space\footnote{By definition, $\mathsf{Ind}^+(v_{n +1}) = |\Sigma_{n +1}^+|$ and $\mathsf{Ind}^+(v_{n + 2}) = 0$.}. \hfill\qed
\end{thm}

For vector fields with symmetry, the Morse Law of Vector Fields has an \emph{equivariant} generalization \cite{K1}.  Here is its brief description: for a compact Lie group $G$ acting on a compact manifold $X$, equipped with a $G$-equivariant field $v$, we prove that the invariants $\{\chi(\d_k^+ X)\}$ can be interpreted as taking values in \emph{the Burnside ring} $\mathcal B(G)$ of the group $G$ (see \cite{D} for the definitions). With this interpretation in place, the appearance of formula \ref{eq2.2} 
does not change.
\smallskip 

Morse formula \ref{eq2.2}  
has an instant, but significant implication:

\begin{cor}\label{cor2.1} Let  $N$ be a smooth neighborhood of the zero set of a vector field $v$ on a compact $(n + 1)$-manifold $X$. Assume that $v$ is boundary generic  with respect to both boundaries, $\d X$ and $\d N$. Then 
$$\mathsf{Ind}(v) = \sum_{j = 0}^{n +1}(-1)^j \chi(\d_j^+N) = \sum_{j = 0}^{n +1} (-1)^j \chi(\d_j^+X).$$ \hfill\qed
\end{cor}

\noindent {\bf Remark 2.3.} Therefore, the numbers $$\sum_{j = 0}^{n +1}(-1)^j \chi(\d_j^+N) \; \, \text{and} \;\,  \sum_{j = 0}^{n +1} (-1)^j \chi(\d_j^+X)$$  can serve as ``more and less localized"  \emph{definitions} of the index invariant $\mathsf{Ind}(v)$. \hfill\qed

An interesting discussion, connected to Theorem \ref{th2.3},  its topological and geometrical implications, can be found in the paper of Gotlieb \cite{Go}.  The ``Topological Gauss-Bonnet Theorem" below is a sample of these results.

\begin{thm}[{\bf Gotlieb}]\label{th2.4} 
Let $X$ be a compact smooth $(n+1)$-dimensional manifold and  $\Phi: X \to \R^{n+1}$ a smooth map which is a \emph{immersion} in the vicinity of  the boundary $\d_1X$. Let $g$ be a Riemannian metric on $X$ which, in the vicinity of $\d X$, is the pull-back $\Phi^\ast(g_E)$ of the Euclidean metric on $\R^{n+1}$.
Consider  a generic linear function $l: \R^{n+1} \to \R$ such that the composite function $f := l\circ\Phi$ has only isolated singularities in the interior of $X$. Let $v := \nabla_g f$ be the gradient  field of $f$\footnote{Thus $v$ is a \emph{transfer} by $\Phi$ of the constant field $\nabla_{g_E}l$.}. Assume that $(f, v)$ is boundary generic.

Then the degree of the Gauss map $$G: \d_1X \to \Phi(\d_1X) \to S^n$$ can be calculated either by integrating over $\d_1X$ the normal curvature $K_\nu$ (in the metric $g$) of the hypersurface $\d_1X \subset X$, or in terms of the  $v$-induced stratification $$\d_1^+X(v) \supset \d_2^+X(v) \supset \dots \supset \d_{n+1}^+X(v)$$ by the formula



\begin{eqnarray}\label{eq2.3}         
\deg(G) = \frac{1}{\textup{vol}(S^n)} \int_{\d_1X} K_\nu\, d\mu_g  = \chi(X) - \mathsf{Ind}(v) \nonumber \\ 
= -\sum_{j = 1}^{n+1}\; (-1)^j \chi(\d_j^+X(v)).
\end{eqnarray}
\footnote{Recall that $\textup{vol}(S^n) =  2\pi^{\frac{n+1}{2}} /\, \Gamma(\frac{n+1}{2})$.}  
\hfill \qed
\end{thm}

\noindent {\bf Example 2.1.} Let $X$ be an orientable surface of genus $g$ with a single boundary component. Let $\Phi: X \to \R^2$ be an immersion, and let $l: \R^2 \to \R$, $f: X \to \R$ and $v := \nabla f$ be as in Theorem \ref{th2.4}. 

Since $\Phi$ is an immersion everywhere (and not only in the vicinity of $\d X$ as Theorem \ref{th2.4} presumes), we get that $v \neq 0$. Thus $\mathsf{Ind}(v) = 0$. Then Theorem \ref{th2.4}  claims that the degree of the Gauss map $G: \d X \to \Phi(\d X) \to S^1$ is equal to $$\chi(X) = 1 - 2g  = \chi(\d_1^+X(v)) -  \chi(\d_2^+X(v))$$ $$= \frac{1}{2}\big(\#(\d_2^-X(v)) - \#(\d_2^+X(v))\big).$$

Thus, the topological Gauss-Bonnet theorem, for immersions $\Phi: X \to \R^2$, reduces to the equation $$\#(\d_2^+X(v)) - \#(\d_2^-X(v)) = 4g - 2.$$

So the number of $v$-trajectories $\g$ in $X$ that are tangent to $\d X$, but are not singletons (they correspond  to points of $\d_2^+X(v)$), as a function of genus $g$,  grows at least as fast as $4g - 2$.
\smallskip

On the other hand, by the Whitney index formula \cite{W1}, the degree of $G: \d X \to S^1$ can be also calculated as $\mu + N^+ - N^-$, where $N^\pm$ denotes the number of positive/negative self-intersections of the curve $\Phi(\d X) \subset \R^2$, and $\mu = \pm 1$.

By a theorem of L. Guth \cite{Gu}, the total number of self-intersections $N^+ + N^- \geq 2g + 2$. Moreover, this lower bound is realized by an immersion $\Phi: X \to \R^2$ ! Therefore, for any immersion $\Phi: X \to \R^2$, the total number of self-intersections of the curve $\Phi(\d X)$ can be estimated in terms of the boundary-tangent $v$-trajectories:
$$N^+ + N^- \geq  3 - \deg(G) = \frac{1}{2}\big(\#(\d_2^+X(v)) - \#(\d_2^-X(v))\big) + 3$$
, and for some special immersion $\Phi$, we get
$$N^+ + N^- =  \frac{1}{2}\big(\#(\d_2^+X(v)) - \#(\d_2^-X(v))\big) + 3.$$
\hfill\qed

\begin{cor}\label{cor2.2} Let $X$ be a compact $(n+1)$-manifold with boundary, which is  properly contained in an open $(n+1)$-manifold $\hat X$.  Let  $\Phi: \hat X \to \R^{n+1}$ be a smooth map which is a \emph{immersion} in the vicinity of  the boundary $\d_1X$. Let $g$ be a Riemannian metric on $\hat X$ which, in the vicinity of $\d_1X$, is the pull-back $\Phi^\ast(g_E)$ of the Euclidean metric on $\R^{n+1}$.

Let  $l: \R^{n+1} \to \R$ be a  linear function, and $f := l\circ\Phi$ its composition with the map $\Phi$. Form the gradient field $v := \nabla_g f$ in $\hat X$.  Assume that the pair $(f, v)$ is boundary generic in the sense of Definition \ref{def2.4}. 

For each $j> 0$, consider a $\e$-small tubular neighborhood $U_j$ of the manifold $\d_jX(v)$ in $\hat X$. Then $\Phi: U_j \to \R^{n+1}$ is an immersion. This setting gives rise to the Gauss map $G_j:  \d U_j \to S^n$,  defined by the formula $G_j(x) = D_x\Phi(\nu_x)/ \| D_x\Phi(\nu_x) \|$, where $x \in \d U_j$ and $\nu_x$ is the unit vector inward normal to $\d U_j$ at $x$. 
\smallskip

Then the degree of the Gauss map $G_j$ can be calculated either by integrating (with respect to the $n$-measure $\mu_g$) over $\d U_j$ the normal curvature $K_\nu$ of the hypersurface $\d U_j \subset \hat X$, or in terms of the  $v$-induced stratum $\d_jX(v)$: 

\makeatletter
\def\tagform@#1{\maketag@@@{(\ignorespaces\theequation\unskip\@@italiccorr)}}
\makeatother

\begin{eqnarray}\label{eq2.4}          
\deg(G_j) = \frac{1}{\textup{vol}(S^n)} \int_{\d U_j} K_\nu\, d\mu_g = \chi(\d_jX(v))  
\end{eqnarray} \hfill \qed
\end{cor}

\begin{proof} We will apply Theorem \ref{th2.4} to the field $v$ in  $U_j$ to conclude that 
$$\deg(G_j) = \frac{1}{\textup{vol}(S^n)} \int_{\d U_j} K_\nu\, d\mu_g  = \chi(U_j) - \mathsf{Ind}(v)$$
Since $v \neq 0$ in $U_j$, $\mathsf{Ind}(v) = 0$, and the last term of this equation reduces to $\chi(U_j) = \chi(\d_jX(v))$.
\end{proof}

\noindent {\bf Remark 2.4.}
Of course, for an odd-dimensional $\d_jX(v)$, the Euler number $\chi(\d_jX(v)) = 0$, and so is $\deg(G_j)$. When $\d_jX(v)$ is even-dimensional (i.e., $n+1 -j = 2l$), the integral in equation \ref{eq2.4} can be expressed in terms of \emph{intrinsic} Riemannian geometry of the manifold $\d_jX(v)$, namely, in terms of the \emph{Pfaffian} $Pf(\Omega)$. The Pfafian is a $2l$-differential form, whose construction utilizes the curvature tensor on the manifold (see \cite{MiS}). So, when $j = n +1 -2l$,  
$$\deg(G_j) = \frac{1}{\textup{vol}(S^n)} \int_{\d U_j} K_\nu\, d\mu_g =  (2\pi)^{-l} \int_{\d_jX(v)} Pf(\Omega)  =  \chi(\d_jX(v))$$
\hfill\qed

\smallskip 

Given a boundary generic field $v$ on $X$, we introduce a sequence of basic degree-type invariants $\{d_k(v)\}$ which are intimately linked, via the Morse formula \ref{eq2.2}, to the invariants $\{\chi(\d_j^+X(v))\}$. 

We use a Riemannian metric $g$ on $X$ to produce the orthogonal projection $v_j$ of the field $v$ on the tangent subspace $T(\d_jX(v)) \subset T(X)$. 

Let $\mathcal S(\d_{k-1}X)$ be the bundle of unit $(n+1 -k)$-spheres associated with the tangent bundle of the manifold $\d_{k-1}X$. We denote by $S(\d_{k-1}X)$ the restriction of the bundle $\mathcal S(\d_{k-1}X) \to \d_{k-1}X$ to the subspace $\d_kX \subset \d_{k-1}X$. 

For each $k$, consider two fields, the inward normal field $\nu_k$ to $\d_kX$ in $\d^+_{k-1}X$ and $v$, as \emph{sections} of the  sphere bundle $p_k: S(\d_{k-1}X) \to \d_kX$ (remember, $v \neq 0$ is tangent to $\d_{k-1}X$ along $\d_kX$ so that $v = v_{k-1}$ along $\d_kX$!). Assume that the sections $v$ and $\nu_k$ are transversal in the space $S(\d_{k-1}X)$. This transversality can be achieved by a perturbation of $\nu_k$ (equivalently, by a perturbation of the metric $g$), supported in the vicinity of the singularity locus $\Sigma_k^+$. Indeed, the intersections occur where the field $v_{k-1}$ is positively proportional to $\nu_k$, that is, where $v_k = 0$. The later locus is exactly the locus $\Sigma_k^+$. The perturbation that does not affect the stratification $\{\d_j^+X\}_j$. Assuming the transversality  of the intersection, the locus $v(\d_kX) \cap \nu_k(\d_kX) \subset S(\d_{k-1}X)$ is zero-dimensional. 

We define the integer $d_k(v) := v \circ \nu_k$ as the algebraic intersection number of two $(n+1-k)$-cycles, $v(\d_kX)$ and $\nu_k(\d_kX)$, in the ambient manifold $S(\d_{k-1}X)$ of dimension $2(n+1-k)$.
\smallskip 


\begin{lem}\label{lem2.2} For a boundary generic field $v$ on a Riemannian manifold $X$, the following formula holds: 
$$d_k(v) = \mathsf{Ind}^+(v_k) = \sum_{j = k}^{n +1} (-1)^j \chi(\d_j^+X).$$
\end{lem}

\begin{proof} We already have noticed that the intersection set $v(\d_kX) \cap \nu_k(\d_kX)$ projects bijectively under the map $p_k: S(\d_{k-1}X) \to \d_kX$ onto the locus $\Sigma_k^+$, where the component $v_k$ of $v$ vanishes and $v$ points inward of $\d^+_{k -1}X$.  It takes more work to see that the sign attached to the transversal intersection point $a \in v(\d_kX) \cap \nu_k(\d_kX)$ is $(-1)^{\mathsf{ind}(p_k(a))}$, where $\mathsf{ind}(p_k(a))$ is the index (the localized degree) of the field $v_k$ in the vicinity of its singularity $p_k(a) \in \Sigma_k^+$. Thus $d_k(v) := v \circ \nu_k = \mathsf{Ind}^+(v_k)$. By the Morse Formula \ref{eq2.2}, the claim of the lemma follows.
\end{proof}

\begin{cor}\label{cor2.3} The integer $d_k(v) = \sum_{j = k}^{n +1} (-1)^j \chi(\d_j^+X)$ depends only on the singular locus $\Sigma^+_k$ of $v_k$ and on the local indices of its points. 
\hfill\qed
\end{cor}

\smallskip
\noindent {\bf Question 2.1.} How to compute $d_j(v)$ in the terms of Riemannian geometry and in the spirit of Theorem \ref{th2.4}  and Corollary \ref{cor2.2}? \hfill\qed
\smallskip

For a boundary generic field $v$ and a fixed metric $g$ on $X$, each manifold $\d_jX(v)$ comes equipped with a preferred normal framing $fr_j$ of the normal bundle $\nu\big(\d_jX(v), \d_1X\big)$: just consider the unitary inward normal field $\nu_1$ of $\d_jX(v)$ in $\d_{j-1}^+X(v)$, then the unitary inward normal field $\nu_2$ of $\d_{j-1}X(v)$ in $\d_{j-2}^+X(v)$,  being restricted to $\d_jX(v)$, then the unitary inward normal field $\nu_3$ of $\d_{j-2}X(v)$ in $\d_{j-3}^+X(v)$, being restricted to $\d_jX(v)$, and so on...  

Via the Pontryagin construction \cite{Po}, this framing $fr_j$ generates a continuous map $G_j(v, g): \d_1X \to S^{j-1}$. Its homotopy class $[G_j(v, g)]$ is an element of the cohomotopy set $\pi^{j-1}(\d_1X)$. If $\d_jX(v) = \emptyset$, then we define $G_j(v, g): \d_1X \to S^{j-1}$ to be the trivial map that takes $\d_1X$ to the base point in $S^{j-1}$. 

Unfortunately, as we will see soon,  $[G_j(v, g)] = 0$! However, when $\d_{j+1}X(v) = \emptyset$, each of the two loci $\d_j^\pm X(v)$ is a closed manifold. Then we can apply the Pontryagin construction only to, say, $\d_j^+X(v)$ to get a map $G^+_j(v, g): \d_1X \to S^{j-1}$. This application leads directly to the following proposition.

\begin{cor}\label{cor2.4}  Consider a boundary generic vector field $v$ such that $\d_{j+1}X(v) = \emptyset$ and a metric $g$, defined in the vicinity of $\d_1X$ in $X$. Then these data give rise to continuous map $G_j^+(v, g): \d_1X \to S^{j-1}$. 

The homotopy class $[G^+_j(v, g)] \in \pi^{j-1}(\d_1X)$ is independent of the choice of $g$ and a homotopy of $v$ within the  open subspace of $\mathcal V^\dagger(X)$, defined by the constraint $\d_{j+1}X(v) = \emptyset$. 

In particular, when $\d_3X(v) = \emptyset$, we get an element $$[G^+_2(v)] \in \pi^1(\d_1X) \approx H^1(\d_1X; \Z)$$, and when $\d_4X(v) = \emptyset$, an element $$[G^+_3(v)] \in \pi^2(\d_1X) \approx H^2(\d_1X; \Z).$$

If $\d_1X = S^n$, we can interpret $[G_j^+(v)]$ also as an element of the homotopy group $\pi_n(S^{j-1})$. \hfill\qed
\end{cor}

The elements $[G_j(v)]$ and $[G^+_j(v)]$ have another classical interpretation as elements of \emph{oriented framed cobordism set} $\Omega_{n-j +1}^{\mathsf{fr}}(\d_1X)$.  In fact, the pair  $(\d_jX(v), fr_j)$ defines the trivial element in $\Omega_{n-j +1}^{\mathsf{fr}}(\d_1X)$. In contrast, if $\d_{j+1}X(v) = \emptyset$, then the bordism class $(\d^+_jX(v), fr_j)$ may be nontrivial. 

Let us recall the definition of framed cobordisms (for example, see \cite{Kos}). Let $M_0, M_1 \subset Y$ be  oriented closed smooth $m$-dimensional submanifolds of a compact $(m+k)$-manifold $Y$, whose normal bundles $\nu(M_0, Y)$ and $\nu(M_1, Y)$ are equipped with framings $fr_0$ and $fr_1$, respectively. 

We say that two pairs $(M_0, fr_0)$  and $(M_1, fr_1)$ define the same element in $\Omega_m^{\mathsf{fr}}(Y)$, if there is a compact $(m+1)$-dimensional oriented submanifold $W \subset Y\times [0, 1]$ whose normal bundle $\nu(W, Y\times [0, 1])$ admits a framing $Fr$ so that: 

(1) $\d W = M_1 \times \{1\} \coprod -M_0 \times \{0\}$,  

(2) the restriction of $Fr$ to $M_1 \times \{1\}$ coincides with $fr_1$, and the restriction of $Fr$ to $M_0 \times \{0\}$ coincides with $fr_0$.

Then the Pontryagin construction establishes a bijection $P: \Omega_m^{\mathsf{fr}}(Y) \to \pi^k(Y)$, where $m + k = \dim Y$. If $m <  k - 1$ both sets admit a structure of abelian groups and the bijection $P$ becomes a group isomorphism.
\bigskip

Now we are in position to explain why $[G_j(v)] = 0$. Consider the obvious embedding $$\a: \d^+_1X(v) \subset \d_1X \times \{0\} \subset \d_1X \times [0, 1].$$ We can isotop $\a$ in $\d_1X \times [0, 1]$ to a regular embedding $$\b: \d^+_1X(v) \subset \d_1X \times [0, 1]$$ such that: 

$(1)$ $\b|_{\d_2X(v)} =  \a|_{\d_2X(v)}$,  and 

$(2)$ the inward normal field $\nu\Big(\b\big(\d_2X(v)\big),\; \b\big(\d_1^+X(v)\big)\Big)$ is parallel to the factor $[0, 1]$ in the product $\d_1X \times [0, 1]$. 

Note that for $j > 2$, all the normal fields $\nu(\d_j X(v), \d_{j-1}^+X(v)$ are preserved under the imbedding $\b$. So, for any $j \geq 2$, the normal framing $fr_j$ of $\a(\d_jX(v))$ in $\a(\d_1X)$ extends to a normal framing $\b(fr_{j-1})$ of $\b(\d^+_{j-1}X)$ in $\d_1X \times [0, 1]$. Therefore $[G_j(v)] = 0$ as an element of the framed bordisms of $\d_1X$. As a result, when $\d_{j+1}X(v) = \emptyset$, we get $[G^+_j(v)] = - [G^-_j(v)]$ in $\Omega_{n-j +1}^{\mathsf{fr}}(\d_1X)$ (equivalently, in $\pi^{j-1}(\d_1X)$).
\bigskip

\section{Deforming the Morse Stratification}

Let $X$ be a smooth compact $(n+1)$-manifold with boundary $\d X$. A boundary generic field $v$ (see  Definition \ref{def2.1}) gives rise to two stratifications \ref{eq2.1}.  

We are going to investigate how the stratification  $\{\d_j^\pm X(v)\}_j$ changes as a result of deforming the vector field $v$.
\smallskip

\begin{lem}\label{lem3.1} Let $N \subset Y$ be a closed  submanifold of a  manifold $Y$ and $M$ a closed manifold. Consider a  family of  maps $\{f_t : M \to Y\}_{t \in [0, 1]}$ such that each $f_t$ is transversal to $N$. All the manifolds, maps, and families of maps are assumed to be smooth. 

Then all the submanifolds $\{f_t(M) \cap N \}$ are isotopic in $N$. In particular, the intersections $f_0(M) \cap N$ and $f_1(M) \cap N$ are diffeomorphic.
\end{lem}

\begin{proof} Let $F: M \times [0, 1] \to Y$ be the map defined by the family $\{f_t\}$. Thanks to  the transversality hypothesis, $F$ is transversal  to $N$ and $F^{-1}(N)$ is a submanifold of $M \times [0, 1]$ whose boundary is $$f_0^{-1}(N) \sqcup f_1^{-1}(N) \subset M \times \d([0, 1]).$$ 
Let $w \neq 0$ be a vector field  on $F^{-1}(N)$, normal to each codimension $1$ submanifold $f_t^{-1}(N)$ in $F^{-1}(N)$. In the construction of $w$, we evidently rely on the property of each $f_t$ being transversal to $N$.  Since $\d(F^{-1}(N)) = f_0^{-1}(N) \sqcup f_1^{-1}(N)$ and $w \neq 0$, each $w$-trajectory that originates at a point of $f_0^{-1}(N)$ must reach $f_1^{-1}(N)$ in finite time. Therefore, employing the $w$-flow, $F^{-1}(N)$ is diffeomorphic to  $f_0^{-1}(N) \times [0, 1]$, and the $F$-image of that product structure in $F^{-1}(N)$ defines a smooth isotopy between $f_0(M) \cap N$ and  $f_1(M) \cap N$ in $N$. This isotopy extends to an ambient isotopy of $N$ itself \cite{Thom}. 

Note that these arguments fail in general if ether $M$ or $N$ have boundaries. However, under additional assumptions (such as  $f_t|_{\d M}$ being $t$-independent and $f_t(M) \cap \d N = \emptyset$), the relative versions of the lemma are valid. 
\end{proof}

\begin{thm}\label{th3.1} The diffeomorphism type of each stratum $\d_j^\pm X(v)$ is constant within each path-connected component of the space $\mathcal V^\dagger(X)$ of boundary generic fields. 
\end{thm}

\begin{proof} If two generic fields, $v_0$ and $v_1$, are connected by a continuous path $v: [0, 1] \to \mathcal V^\dagger(X)$, then they can be connected by a path $\tilde v: [0, 1] \to \mathcal V^\dagger(X)$ such that the dependence of the field $\tilde v(t)$ on $t \in [0,1]$ is smooth. The argument is based on the property of generic fields to form an open set in the space of all fields (Theorem \ref{th2.1}), the smooth partition of unity technique (which utilizes the compactness of manifold $X \times [0,  1]$), and the standard  techniques of approximating continuos functions with the smooth ones.

Thus it suffices to consider a smooth $1$-parameter family of vector fields $v_t \in \mathcal V^\dagger(X)$, connecting $v_0$ to $v_1$. Since  any generic field $v_t|_{\d_1X}$, viewed as a section of the vector bundle $\eta_1: TX|_{\d_1X} \to \d_1X$, is transversal its zero section, we may apply Lemma \ref{lem3.1} (with $M = \d_1X$, $N$ being the zero section of $\eta$, $Y = E(\eta_1)$, and $f_t = v_t$) to conclude that all the submanifolds ${\d_2X(v_t)}_t$ are isotopic in $\d_1X$. 

Since each $\d_2X(v_t)$ divides $\d_1X$ into a pair of complementary domains, $\d_1^+X(v_t)$ and $\d_1^-X(v_t)$, and since their polarity $\pm$ is determined by the inward/outward direction of $v_t$, which changes continuously with $t$, the ambient isotopy of $\d_1X$ (which takes $\d_2X(v_0)$ to $\d_2X(v_t)$) must take $\d_1^+X(v_0)$ to $\d_1^+X(v_t)$. The isotopy $h_t: \d_1X \to \d_1X$ extends to an isotopy $\tilde h_t: X \to X$.

A similar argument applies to lower strata $\d_j^\pm X(v_t)$. Indeed, with the isotopy $h_t: \d_1X \to \d_1X$ that takes $\d_2X(v_0)$ to $\d_2X(v_t)$ in place, consider the two sections, $v_0$ and $(h_t^{-1})_\ast(v_t)$, of the bundle $\eta_2: T(\d_1X)|_{ \d_2X(v_0)} \to \d_2X(v_0)$, both sections being transversal to the zero section of $\eta_2$.
Applying again Lemma \ref{lem3.1}, we conclude that the loci $\d_3X(v_0)$ and $h_t^{-1}(\d_3X(v_t))$ are isotopic in $\d_2X(v_0)$ (recall that these loci are exactly the transversal intersections of two sections $v_0$ and $(h_t^{-1})_\ast(v_t)$ of $\eta_2$ with its zero section). Again, an isotopy $h'_t: \d_2X(v_0) \to \d_2X(v_0)$ that takes $\d_3X(v_0)$ to $h_t^{-1}(\d_3X(v_t))$ must take $\d_2^+X(v_0)$ to $h_t^{-1}(\d_2^+X(v_t))$. The isotopy $h'_t$ extends to an isotopy $\tilde h'_t: X \to X$ which preserves the pair $\d_2X(v_0) \subset \d_1X$.  So, the pairs $\d_3X(v_0) \subset \d_2^+X(v_0)$ and $\d_3X(v_t) \subset\d_2^+X(v_t)$ are diffeomorphic via the composite isotopy $\tilde h'_t \circ \tilde h_t$.

This reasoning can be recycled to prove that all the pairs $\d_j^+X(v_0)$ and $\d_j^+X(v_t)$ are diffeomorphic via a \emph{single} isotopy of $X$. This argument will be carried explicitly in the proof of  Theorem 3.4 from \cite{K3}. 
\end{proof}

\begin{cor}\label{cor3.1} Let $X$ be a $(n+1)$-dimensional compact smooth manifold with boundary. 

Within each path-connected component of the space $\mathcal V^\dagger(X)$ of generic fields, the numbers  $\{d_k(v)\}_{0 \leq k \leq n}$, as well as the numbers $\{\chi(\d_k^\pm X(v))\}_{1 \leq k \leq n+1}$, are constant.
\end{cor}

\begin{proof} The claim follows instantly from Theorem \ref{th3.1} and Lemma \ref{lem2.2}. 
\end{proof}

For a manifold $X$ with nonempty boundary, by deforming any given function $f: X \to \R$ and its gradient-like field $v$, we can expel the isolated $v$-singularities from $X$. This can be achieved by the appropriate ``finger moves" which originate at points of the boundary $\d X$ and engulf the isolated singularities of $v$. The result of these manipulations lead to
 
\begin{lem}\label{lem3.2} Any $(n + 1)$-manifold $X$ with a non-empty boundary admits a Morse function $f: X \rightarrow \R$ with no critical points in the interior of $X$ and such that $f|: \d X \rightarrow \R$ is a Morse function. Such functions form an open nonempty set in the space $C^\infty(X)$ of all smooth functions on $X$. 

As a result, the gradient-like vector fields $v \neq 0$ on $X$ form an open nonempty set in the space $\mathcal V(X)$ of all all vector fields on $X$.
\end{lem}

\begin{proof} Let us sketch the main idea of the argument. Start with a Morse function $f: X \to \R$. Connect each critical point in the interior of $X$ by a smooth path to a point on the boundary in such a way that a system of non-intersecting paths is generated. Then delete from $X$ small regular neighborhoods of those paths (``dig a system of dead-end  tunnels") and restrict $f$ to the remaining portion $X^\odot$ of $X$. Smoothen the entrances  of the tunnels so that the boundary of $X^\odot$ will be a smooth manifold which is diffeomorphic to $X$. We got a nonsingular function $f$ on $X^\odot$. A slight perturbation of $f$ on $X^\odot$ will not introduce critical points in the interior of $X^\odot$ and will deliver a Morse function on its boundary.  Indeed, recall that  the sets of Morse functions on $X$ and $\d X$ are open and dense in the spaces $C^\infty(X)$ and $C^\infty(\d X)$ of all smooth functions, respectively (for example, see \cite{GG}). 

Of course $v \neq 0$ is an open condition imposed on a vector field on a compact manifold. On the other hand, if $df(v) > 0$, then any field $v'$, sufficiently close to $v$, will have the property $df(v') > 0$. The previous arguments show that the set of gradient-like non-vanishing fields is nonempty. So it is an open nonempty subspace in the space $\mathcal V(X)$ of all all vector fields on $X$.  
\end{proof}

Eliminating isolated critical points  of a given function $f: X \to \R$ on a manifold with boundary is not ``a free lunch": the elimination introduces new critical points of the restricted function $f: \d X \to \R$. This is a persistent theme throughout our program: 

\emph{Expelling critical points of  gradient flows from a manifold $X$ leaves crucial residual geometry on its boundary}.  

This boundary-confined geometry allows for a reconstruction of the topology of $X$. 

Ideas like these will be developed in the future papers from this series.  Meanwhile, the following lemma gives a taste of  things to come.

\begin{lem}\label{lem3.3} Let $f: X \to \R$ be a Morse function with no local \emph{extrema} in the interior of a $(n +1)$-manifold $X$. Then an elimination by a finger move\footnote{as in the proof of Lemma \ref{lem3.2}} of each $f$-critical point $p_\star$ of the Morse index $i(p_\star)$ results in the introduction of $[2(n + 1 - i(p_\star)) -1]$ new critical points of positive type and $2i(p_\star) + 1$ new critical points of negative type for the modified function $f|_{\d X}$.
\end{lem}

\begin{proof} Let  $p_\star$ be a Morse singularity of $f$ in the interior of $X$. Denote by $S_{p_\star}$ a sphere which bounds a small disk $D_{p_\star}$ centered on $p_\star$ and such that $f|_{S_{p_\star}}$ is a Morse function. Without loss of generality, we can assume that, in the Morse coordinates $\{x_i\}$, $S_{p_\star}$ is given by $\sum_{i = 1}^{n +1} x_i^2 = 1$, while $f(x) =  \sum_{i = 1}^{n +1} a_i x_i^2$ with all the $\{a_i \neq 0\}$ being \emph{distinct}. Then $f|_{S_{p_\star}}$ has only Morse-type singularities at the points where the coordinate axes pierce the sphere $S_{p_\star}$. With respect to the pair $(X \setminus D_{p_\star}, f)$, these points come in two flavors: positive and negative. The two types are separated by the hypersurface of the cone $$C = \big\{\sum_{i = 1}^{n +1} a_i x_i^2 = 0\big\}.$$ In the vicinity of $p_\star$, the intersection $C \cap S_{p_\star}$ is exactly the locus $$\d_2(X \setminus D_{p_\star}) = \d_2^+(X \setminus D_{p_\star})$$, so that the $f$-gradient field $v$ (tangent to $S_{p_\star}$ along $C \cap S_{p_\star}$) is transversal to $C \cap S_{p_\star}$, the product of two spheres.  Therefore, in the vicinity of $x_\star$, $\d_3(X \setminus D_{p_\star}) = \emptyset$!

The function $f|_{S_{p_\star}}$ has exactly $2 \cdot i(p_\star)$ critical points of the positive type and exactly $2(n + 1 - i(p_\star))$ critical points of the negative type. We shall denote these sets by $\Sigma^{\pm}_1(p_\star)$ and the two domains in which $C$ divides $S_{p_\star}$---by $S^{\pm}_{p_\star}$. 

Let $x \in S^+_{p_\star}$ be a local maximum of $f|_{S_{p_\star}}$. Note that it is possible to connect  $x$ to a non-singular (for $f|_{\d X}$) point $y \in \d X$ by a smooth path $\g$ along which $f$ is increasing. Indeed, any non-extendable path $\g$ such that $df(\dot \g) > 0$ either approaches a critical point or reaches the boundary $\d X$. By a small perturbation, we can insure that $\g$ avoids all the (hyperbolic) critical points in the interior of $X$ (by the hypothesis, $f$ has no local maxima/minima in the interior of $X$). Thus $\g$ can be extended until it reaches the boundary $\d X$ at a point $y$. 

Drilling a narrow tunnel $U$ diffeomorphic to the product $\g \times D^n$ along $\gamma$ does not change the topology of $X$; the function $f|_{X \setminus U}$ retains almost the same list of singularities at the boundary as the function $f|_{X \setminus D_{p_\star}}$ has:  more accurately, the local maximum at $x \in S^+_{p_\star}$ disappears in $\d(X \setminus U)$ and a \emph{negative} critical point of index 1 of $f|_{\d(X \setminus U)}$ appears near the $y$-end of the tunnel $U$. Thus we have modified $f$ and have eliminated the critical point $p_\star$ in the interior of $X$ at the cost of introducing on the boundary $2(n + 1 - i(p_\star)) -1$ critical points of positive type and $2i(p_\star) + 1$ critical points of negative type. 
\end{proof}

Soon, motivated by Lemma \ref{lem3.2}, we will restrict our attention to \emph{nonsingular} functions $f: X \to \R$ and their gradient-like fields $v$---an open subset in the space of all gradient-like pairs $(f, v)$; but for now, let us investigate a more general case.
\smallskip

Consider Morse data $(f, v)$, where the field $v$ is nonsingular along the boundary $\d_1X$. Extend $(f, v)$ to $\hat X: = X \cup C$ and $\hat v$, where $C$ is some external collar of  $\d_1X$ so that the extension $(\hat f, \hat v)$ is nonsingular in $C$.  At each point $x \in \d_1X$, the $\hat v$-flow defines a projection $p_x$ of the germ of $\d_1X$ into the germ of the hypersurface $\hat f^{-1}(\hat f(x))$. 

Let $\d_jX^\circ$ and $\d_j^\pm X^\circ$ denote the pure strata $\d_jX \setminus \d_{j+1}X$ and $\d_j^\pm X \setminus \d_{j+1}X$, respectively.  At the points $x \in \d_1X^\circ$, $p_x$ is a surjection; at the points of $x \in \d_2X^\circ$,  it is a folding map; at the points $x \in \d_3X^\circ$, it is a cuspidal map. Often we will refer to points $x \in \d_1 X$  by the smooth types  of their $p_x$-projections.
\smallskip

As the theorem and the corollary below testify, for a given function $f: X \to \R$, we enjoy a considerable freedom in changing the given Morse stratification $\{\d_j^+X := \d_j^+X(v)\}$ by  deforming  the $f$-gradient-like field $v$ (cf. Section 3 in \cite{K}).  

\begin{thm}\label{th3.2} Let $X$ be a compact smooth $(n + 1)$-manifold with nonempty boundary.  Take a smooth function $f: X \to \R$ with no singularities along  $\d X$, and let $v$ be its gradient-like field. Consider a stratification $$X := Y_0 \supset Y_1 \supset Y_2 \supset \dots \supset Y_{n + 1}$$ of $X$ by compact smooth manifolds $\{Y_j\}$, and let $S_j$ and $S_j^\d$ denote the critical sets of the restrictions $f|_{Y_j}$ and $f|_{\d Y_j}$, respectively. Assume that the following properties are satisfied:
\begin{itemize}
\item  $\dim(Y_j) = n + 1 - j$,  
\item $Y_1 \subset \d X$ and $\{Y_j \subset \d Y_{j -1}\}$ are regular embeddings for all $j \in [2, n+1]$,  
\item for each $j \leq n + 1$ the functions $f|_{Y_j}$ and $f|_{\d Y_j}$ have Morse-type critical points at the loci $S_j$ and $S_j^\d$, respectively, 
\item at the points of $S_j$, $df(\nu) > 0$ and, at the points of $S_{j - 1}^\d \setminus S_j$,  $df(\nu) < 0$, where 
$\nu$ is the inward normal to $\d Y_{j-1}$ in $Y_{j-1}$\footnote{This condition is metric-independent: it does not depend on the choice of $\nu$.}.
\end{itemize}
Then, within the space of $f$-gradient-like fields, there is a deformation of $v$  into a new  boundary generic gradient-like field $\tilde v$, such that the stratification $\{\d_j^+X(\tilde v)\}_{0 \leq j \leq n+1}$, defined by $\tilde v$, coincides with the given stratification $\{Y_j\}_{0 \leq j \leq n + 1}$. 
\end{thm}

\begin{figure}[ht]
\centerline{\includegraphics[height=2in,width=3in]{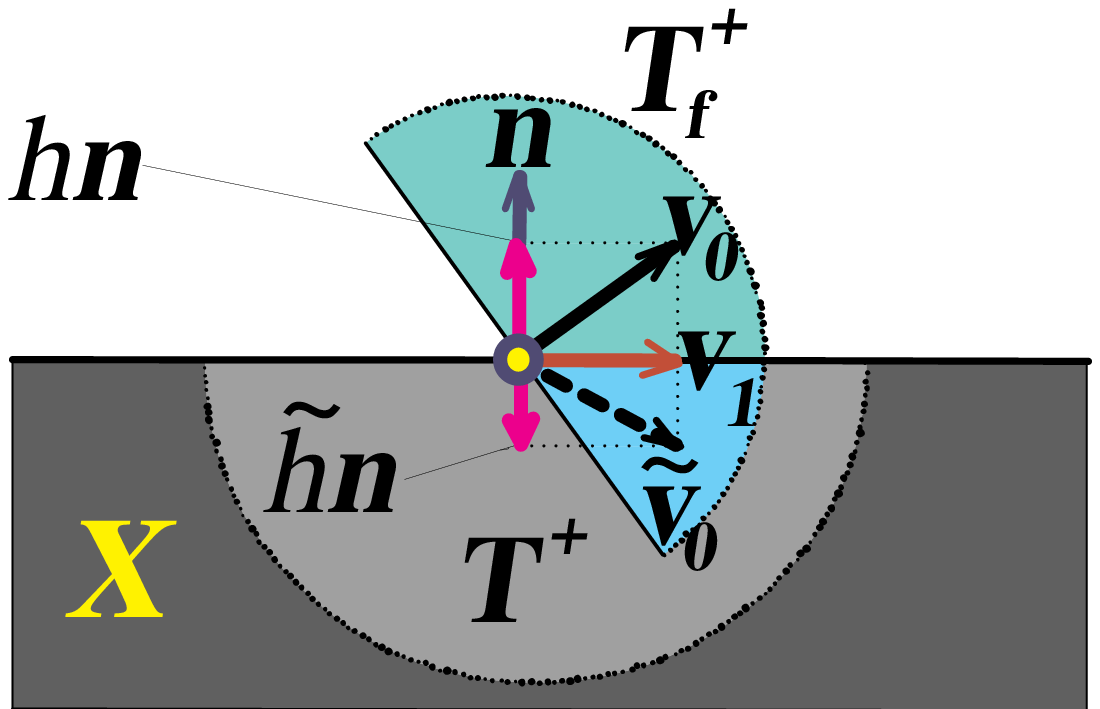}} 
\bigskip
\caption{}
\end{figure}

\begin{proof} 

We pick a Riemannian metric $g$ in a collar $U$ of $\d X$ in $X$ so that $v$ becomes the gradient field of $f$.  Consider auxiliary vector fields $\{v_j\}$, where $v_j$ denotes the orthogonal projection of $v$ on the tangent spaces of  closed manifold $\d Y_{j-1}$.

The construction of the desired field $\tilde v$ is inductive in nature, the induction being executed in increasing values of the index $k$. Fig. 4 illustrates a typical inductive step.

Assume that $v := \tilde v$ has been already constructed so that $\d_j^+X(v) = Y_j$ and $\Sigma_j^+(v) = S_j$ for all $j < k$. This assumption implies that  $v$ is tangent to $Y_j$ exactly along its boundary $\d Y_j$ for all $j < k$. Along $\d Y_{k - 1} = \d(\d_{k-1}^+X(v)) = \d_k X(v)$ (and thus along $Y_k \subset \d Y_{k - 1}$), we decompose $v$ as $v_k + \sum_{j = 0}^{k - 1} n_j$, where $n_j := v_{j - 1} - v_j$. 

The idea is to modify $v$ in the direction normal to $\d_kX(v)$ in $\d_{k-1}X(v)$, while keeping the rest of  its components $\{n_j\}$ unchanged.  

Denote by $T_x$ the tangent space of $Y_{k-1}$ at $x \in \d Y_{k-1}$. Let $T_x^+$ be the open half-space of $T_x$ positively spanned by the vectors that point inside of $Y_{k-1}$.  Let $T_x^+(f)$ be half  of the tangent space $T_x$, defined by $df(u) > 0$, where $u \in T_x$.  We introduce the complementary to $T_x^+$ and $T_x^+(f)$ open half-spaces $T_x^-$ and  $T_x^-(f)$.  

At each point $x \in Y_k$, consider the open cone $C_x^+ = T_x^+ \cap T_x^+(f)$ and, at each point $x \in \d Y_{k-1}\setminus Y_k$, the open cone $C_x^- = T_x^- \cap T_x^+(f)$ (see Fig. 4). These cones are non-empty, except perhaps at the points of $S_{k-1}^\d$, where $\pm v_{k - 1}$ is anti-parallel to the inward normal $\nu_k$ of $\d Y_{k-1} \subset Y_{k-1}$. However, at  $x \in S_k$, $C_x^+ \neq \emptyset$, and  at $x \in S^\d_{k-1} \setminus S_k$, $C_x^- \neq \emptyset$ due to the last bullet in the hypotheses of the theorem. Thus, for each $x \in Y_k^\circ$, there is a number $h$ so that the vector $u_k =  v_k + h\cdot \nu_k \in C_x^+$ (this conclusion uses the the property $df(v_k) > 0$ on the set $Y_k \setminus S_k$). Similarly,  for each $x \in \d Y_{k-1} \setminus Y_k$, there is a number $h$ so that $u_k  \in C_x^-$.  By the partition of unity argument, which employes convexity of the cones $C_x^\pm$,  there is a smooth function $h: \d Y_{k-1} \to \R$ which delivers the desired field $u_k$ along $\d Y_{k}$.  In order to insure the continuity of $h$ and $u_k$ across the boundary $\d Y_k \subset \d Y_{k - 1}$,  we require  $h|_{\d Y_k} = 0$.  Thus $u_k = v_k \neq 0$ on $\d Y_k$.

Put $v' = u_k + \sum_{j = 0}^{k -1} n_j $.  Now, $\d_j^+X(v') = \d_j^+X(v) = Y_j$ for all $j < k$ (these strata depend on the $n_j$'s only), and $\d_k^+X(v') = Y_k$ by the construction of $u_k$. Moreover, $\Sigma_j^+(v') =  \Sigma_j^+(v) = S_j$ for all $j \leq k$. In fact, $v'$ is tangent to $Y_{k-1}$ along $\d Y_{k-1}$. Note that this inductive argument should be modified for $k = n + 1$ since   $Y_{n+1} = S_{n+1}$ is 0-dimensional.

We smoothly extend $v'$ into a regular neighborhood $V$ of $\d Y_{k-1}$ in $X$. Abusing notations, we denote this extension by $v'$ as well. The neighborhood $V$ is chosen so that there $df(v') > 0$.

To complete the proof of the inductive step $k - 1 \Rightarrow k$, we form the field $\tilde v := \psi_0 v + \psi_1 v'$, where the functions $\{\psi_0, \psi_1\}$ deliver a smooth partition of unity subordinate to the cover $\{X \setminus V, V\}$ of $X$. Since $df(\sim) > 0$ defines a convex cone in the space of vector fields, $\tilde v$  is a $f$-gradient-like field with the desired Morse  stratification.
\end{proof}

Theorem \ref{th3.2} has an immediate implication:

\begin{cor}\label{cor3.2} Let $f: X \to \R$ be a Morse function and $v$ its boundary generic gradient-like field with the Morse stratification $\{\d_j^+X(v) \subset \d_jX(v)\}_{0 \leq j \leq n+1}$. Assume that compact codimension zero submanifolds  $Y_j \subset \d_j X$ are chosen so that, for each $j$,   $Y_j \supset \Sigma_j^+(v)$ and $Y_j \cap \Sigma_j^-(v) = \emptyset$.

Then, within the space of $f$-gradient-like fields,  it is possible to deform $v$  into a new gradient-like boundary generic field $\tilde v$, such that the stratification $\{\d_j^+X(\tilde v)\}_{0 \leq j \leq n +1}$  coincides with the given stratification $\{Y_j\}_{0 \leq j \leq n+ 1}$.   Moreover, $\{\d_jX(\tilde v) = \d_jX(v)\}_{0 \leq j \leq n + 1}$. 

In particular, if $\Sigma_k^+(v) = \emptyset$, the claim is valid for any stratification $\{Y_j\}_{0 \leq j \leq n +1}$ as above that terminates with $Y_{k} = \emptyset$. \hfill\qed
\end{cor}

The next proposition (based on Corollary \ref{cor3.2}) shows that, for a given Morse function $f: X \to \R$, by an appropriate choice of gradient-like field $v$, the Morse stratification $\d_j^+X$ can be made topologically very simple and regular: namely, each stratum $\d_j^+X$ is a disjoint union of $(n+1-j)$-dimensional disks. Moreover, when the boundary $\d_1X$ is connected and $j \in [1, n-1]$, each stratum $\d_j^+X$ is a just a single disk.

\begin{cor}\label{cor3.3} Let $f: X \to \R$ be a Morse function on a compact $(n + 1)$-manifold $X$, $f$ being nonsingular along the boundary $\d_1X$. We divide the connected components $\{\d_1X_\a\}_\a$ of the boundary into two types, $\mathsf A$ and $\mathsf B$.  By definition, for type $\mathsf A$, the singularity set $\Sigma^+_1(f) \cap \d_1X_\a \neq \emptyset$, and for type $\mathsf B$, $\Sigma^+_1(f) \cap \d_1X_\a = \emptyset$. 

Then any $f$-gradient-like field $v$ can be deformed, within the space of $f$-gradient-like fields, into a boundary generic field $\tilde v$ so that,  for each component $ \d_1X_\a$ of type $\mathsf A$ and all $j < n$, the  stratum $\d_j^+X(\tilde v) \cap \d_1X_\a$ is diffeomorphic to a disk $D^{n+1- j}$. At the same time, for the components of type $\mathsf B$  and all $j \geq 1$, the  stratum $\d_j^+X(\tilde v) \cap \d_1X_\a = \emptyset$.

For the components of type $\mathsf A$, in contrast, the $1$-dimensional stratum $\d_{n}^+X(\tilde v) \cap \d_1X_\a$ is a finite union of arcs residing in the circle $\d_nX(\tilde v) \cap \d_1X_\a$. 
Moreover, $\chi(\d_n^+X(\tilde v))$, the number or arcs in $\d_n^+X(\tilde v)$, and the number of points in $\d_{n + 1}^+X(\tilde v)$ are linked via the formula  
$$|\d_{n+1}^+X(\tilde v)| = \chi(\d_n^+X(\tilde v)) + (-1)^{n+1}[ \mathsf{Ind}(v) -\chi(X)] + \frac{m}{2}[(-1)^{n +1} -1]$$
, where $\mathsf{Ind}(v) = \mathsf{Ind}(\tilde v)$ is the index of the field $v$, and $m$ is the number of boundary components of type $\mathsf A$. \hfill\qed
\end{cor}
\begin{proof} If $n \geq 2$, for each type $\mathsf A$ connected component $\d_1X_\a$ of $\d_1X$, the singularity set $\Sigma_1^+(f) \cap \d_1X_\a$  can be included in a disk $D^n_\a \subset \d_1^+X(v)$. By Corollary \ref{cor3.2}, we can deform $v$ to a new $f$-gradient-like field $v'$, so that the new stratum $\d_1^+X(v') \cap  \d_1X_\a = D^n_\a$. If $n \geq 3$, then the singularity set $\Sigma^+(f|_{\d D^n_\a})$ can be incapsulated in a disk $D^{n-1}_\a$. By the same token, after still another deformation $v''$ of $v'$, we can arrange  for $\d_1^+X(v'') \cap  \d_1X_\a = D^n_\a$ and $\d_2^+X(v'') \cap  \d_1X_\a = D^{n-1}_\a$. This process repeats itself, unless the dimension of $\d_j^+X(\tilde v) \cap \d_1X_\a$ becomes one. At its final stage,  $\d_n^+X(\tilde v) \cap \d_1X_\a$ consists of several arcs   which are contained in the circle $\d D^2_\a$.

For each type $\mathsf B$ connected component $\d_1X_\a$ of $\d_1X$, by a similar reasoning, we can arrange for  $\d_1^+X(\tilde v)_\a = \emptyset$. Thus, $\d_j^+X(\tilde v) \cap \d_1X_\a = \emptyset$ for all $j \geq 1$ and $\a \in \mathsf B$.

Therefore, letting $Y_j = \coprod_{\a \in \mathsf A} D^{n+1-j}_\a$  for all $j \in [1, n-1]$ in Corollary \ref{cor3.2}, we have established all the claims of the corollary, but the last one.

Since $v$ and $\tilde v$ both are the gradient-like fields for the same Morse function $f$, their indexes, $\mathsf{Ind}(v)$ and $\mathsf{Ind}(\tilde v)$, are equal. Thus we get $$\mathsf{Ind}(v) = \mathsf{Ind}(\tilde v) = \chi(X) + \frac{m}{2}[(-1)^{n +1} -1] + (-1)^n\big[ \chi(\d_n^+X(\tilde v)) - \chi(\d_{n+1}^+X(\tilde v))\big]$$, where $\frac{m}{2}[(-1)^{n +1} -1]$ is the contribution of all the disk-shaped strata $\{\d_j^+X(\tilde v)\}_{1 \leq  j  < n}$ to the Morse formula \ref{eq2.2}. 
\end{proof}
\smallskip

Recall that, by Corollary 4.4 \cite{K}, for any 3-fold $X$ and a  boundary generic field $v \neq 0$ on it, we get $|\d_3^+X(v)| \geq 2\chi(X) - 2$, provided $\d_1^+X(v) \approx D^2$. Thus, as a positive $\chi(X)$ increases, the boundary of the disk $\d_1^+X(v)$ becomes more ``wavily". 

If $X$ is the Poincar\'{e} homology $3$-sphere with a $3$-ball being deleted, then by Corollary 4.4 \cite{K}, $|\d_3^+X(v)| > 0$ for any gradient-like field $v \neq 0$ such that $\d_1^+X(v) \approx D^2$.
\smallskip

These examples motivate 
\smallskip

\noindent {\bf Question 3.1} For boundary generic gradient-like fields $v$ with a fixed value $i$ of the index $\mathsf{Ind}(v)$ and a disk-shaped stratification $\{\d_j^+X(v)\}_{1 \leq  j  < n}$ as in Corollary \ref{cor3.3}, what is the minimum $\mu(X, i)$ of $|\d_{n + 1}^+X(v)|$?  \hfill\qed

\smallskip

Evidently, such number $\mu(X, i )$ is an invariant of  the diffeomorphism type of $X$. It seems that $\mu(X, i)$ is semi-additive under the connected sum operation: that is, $$\mu(X_1 \# X_2, i_1 + i_2) \leq \mu(X_1, i_1) + \mu(X_2, i_2).$$ 

\section{Boundary Convexity and Concavity of Vector Fields}

We are ready to introduce pivotal concepts of the stratified convexity and concavity for smooth vector fields on manifolds with boundary.

\begin{defn}\label{def4.1} Given a boundary generic vector field $v$ (see Definition \ref{def2.1}), we say that $v$ is boundary $s$-\emph{convex}, if $\d_s^+X = \emptyset$. In particular, if $\d_2^+X = \emptyset$, we say that $v$ is boundary $2$-convex, or just \emph{boundary convex}.  

We say that $v$ is boundary  $s$-\emph{concave}, if $\d_s^-X = \emptyset$. In particular, if $\d_2^-X = \emptyset$, we say that $v$ is boundary $2$-concave, or just \emph{boundary concave}. 
 \hfill \qed
\end{defn}

\noindent {\bf Example 4.1.} Assume that a compact manifold $X$ is defined as a 0-dimensional submanifold in the interior of a Riemannian manifold $Y$, given by an inequality $\{x: h(x) \geq 0\}$, where  $h: Y \rightarrow \R$ is a smooth function with 0 being a regular value.  Then the boundary convexity of a gradient field $v := \nabla f$  in $X$ can be expressed in terms of the Hessian matrix $Hess(h)$ by the inequality $$\langle Hess_x(h)v(x), \; v(x) \rangle < 0$$ at all points $x$, where $v(x)$ is tangent to $\d X$. If $$\langle Hess_x(h)v(x),\; v(x) \rangle >  0$$,  where $v(x)$ is tangent to $\d X$, then the field $v$ is boundary concave. \hfill \qed
\smallskip

\noindent {\bf Example 4.2.} According to the argument in Lemma \ref{lem3.3}, the complement to a small convex (in the Morse coordinates) disk, centered at a Morse type $f$-critical point, is \emph{boundary concave}  with respect to the gradient field $v = \nabla f$. In fact, the field $v$ is both boundary $3$-concave and $3$-convex! So, if $f: Y \to \R$ is a Morse function on a \emph{closed} manifold $Y$ with a critical set $\Sigma$, then the complement $X$ in $Y$ to a small locally convex neighborhood of $\Sigma$ admits a boundary concave $f|_{X}$-gradient-like field (with $\d_3X = \emptyset$)! 
\hfill \qed
\smallskip

Theorem \ref{th4.1} below belongs to a family of results which we call ``\emph{holographic}" (see also and Theorem \ref{th4.2}). The intension in such results is to reconstruct some structures on the ``bulk" $X$ (or even the space $X$ itself) from the appropriate flow-generated structures (``observables") on its boundary $\d X$. A paper from this series will be devoted entirely to the phenomenon of holography for nonsingular gradient flows. 

In Theorem \ref{th4.1},  we describe how some \emph{boundary-confined} interactions between the critical points of a given function $f: \d_1X \to \R$ of opposite polarities can serve as an indicator of the convexity/concavity of the gradient field $\nabla f$ in $X$ (recall that the convexity/concavity properties of the $v$-flow do require knowing the field in the \emph{vicinity} of $\d_1X$ in $X$!). 

\begin{thm}\label{th4.1} Let $f: X \rightarrow \R$, $f|: \d_1X \to \R$ be Morse functions and $v$ and $v_1$ their gradient  fields with respect to a Riemannian metric $g$ on $X$ and its restriction to $\d_1X$, respectively. Assume that $v$ is boundary generic.  

If $\d_2^\pm X(v) = \emptyset$, then there is no \emph{ascending} $v_1$-trajectory $\g: \R \to \d _1X$, such that $$\lim_{t \rightarrow -\infty} \g(t) \in \Sigma_1^\mp \;\; \text{and} \; \; \lim_{t \rightarrow +\infty} \g(t) \in \Sigma_1^\pm$$ (both critical sets $\Sigma_1^\pm$ depend only on $f$).

Conversely, if for a given $f$-gradient pair $(v, v_1)$, no such $v_1$-trajectory $\g \subset \d_1X$ exists, then one can deform $(v, v_1)$  to a new boundary generic pair $(\tilde v, \tilde v_1)$ of the $f$-gradient type so that $\d_2^\pm X(\tilde v) = \emptyset$. Moreover, the fields $v_1$ and $\tilde v_1$ on $\d_1X$ can be chosen to be arbitrary close in the $C^\infty$-topology. 

In particular, if $f(\Sigma_1^+) < f(\Sigma_1^-)$ (as sets),  then $X$ admits a boundary generic and convex $f$-gradient-like field $\tilde v$; similarly, if $f(\Sigma_1^+) > f(\Sigma_1^-)$,  then $X$ admits a boundary generic and concave $f$-gradient-like field $\tilde v$.
\end{thm}

\begin{proof} First consider the convex case, that is,  the relation between the property $\d_2^+ X(v) = \emptyset$ and the absence of an ascending $v_1$-trajectory $\g: \R \to \d _1X$ which connects  $\Sigma_1^- $ to $\Sigma_1^+$.

Consider the function $h: \d_1 X \rightarrow \R$, defined via the formula $v = v_1 + h \cdot n$, where $n$ denotes a unitary field inward normal  to $\d_1X$ in $X$. Since $v$ is boundary generic, zero is a regular value of $h$.  Then $$\d^+_1X(v) = h^{-1}([0, +\infty)), \quad \d^-_1X(v) = h^{-1}((-\infty, 0])$$ and $\d_2X(v) = h^{-1}(0)$. 

\begin{figure}[ht]\label{fig.5}
\centerline{\includegraphics[height=1.7in,width=3.8in]{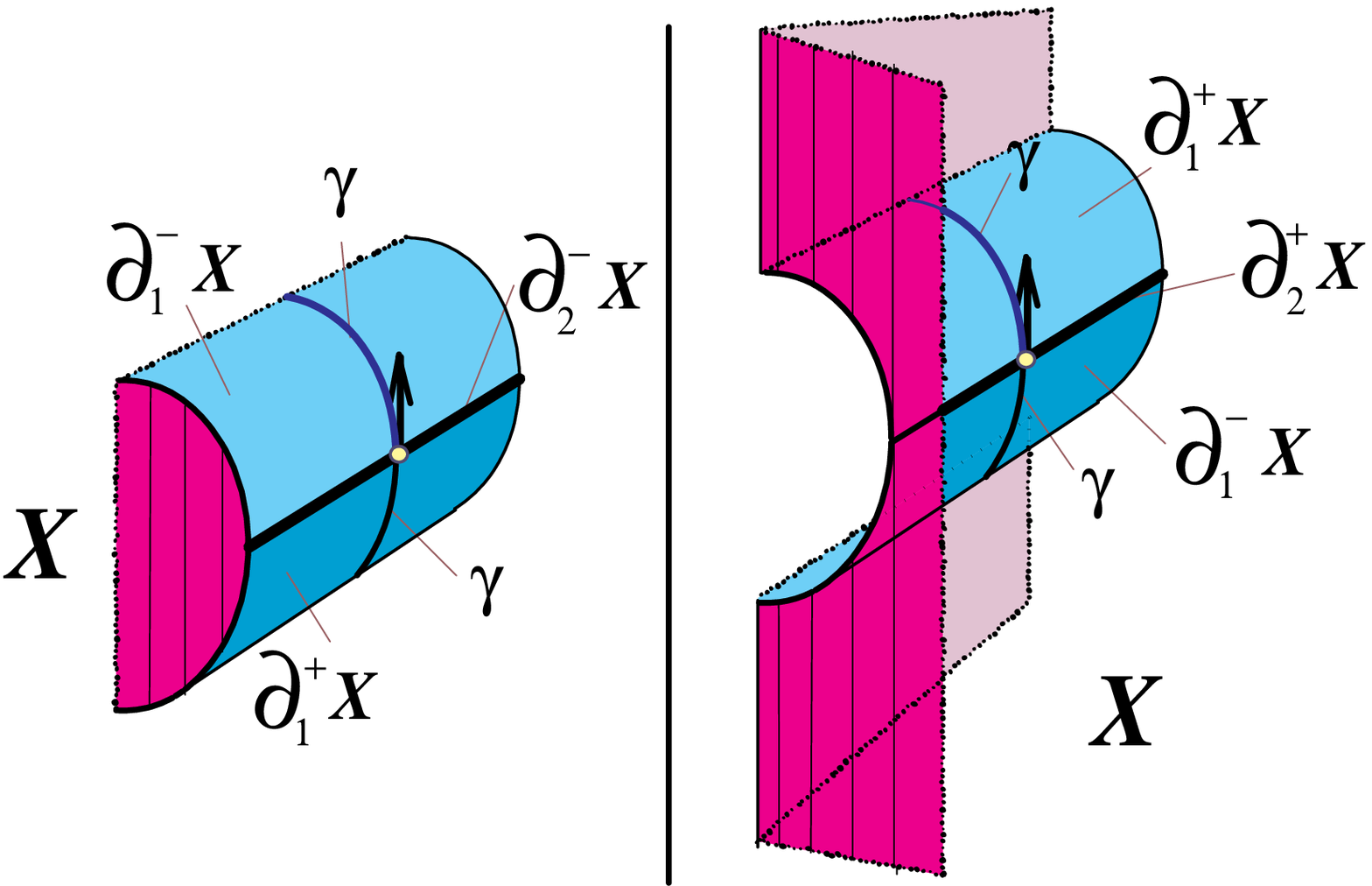}} 
\bigskip
\caption{}
\end{figure}

If an ascending $v_1$-trajectory $\g \subset \d_1X$, which links  $\Sigma_1^-$ with $\Sigma_1^+$, does exist, it must cross somewhere the boundary $\d_2X(v)$ of $\d^-_1X(v)$. Since the field $v_1$ is an orthogonal projection of $v$ on $T(\d_1X)$, the two fields must agree at any point  $x \in \d_2X(v)$---the locus where $v$ is tangent to $\d_1X$. Thus, $v_1$ is the gradient of $f$ at $x \in \g \cap \d_2X(v)$. Therefore, as $\g(t)$ crosses from $\d_1^+X(v)$ into $\d_1^-X(v)$ at $x$, in its vicinity, the arc $\g \cap\d_1^-X(v)$ lies \emph{below} the arc $\g \cap \d_1^+X(v)$ (see Fig. 5). By the definition of the locus $\d_2^+X(v)$, such crossing $x \in \g \cap\d_2X(v)$ belongs to $\d_2^+X(v)$. Therefore, $\d_2^+X(v) \neq \emptyset$, contrary to the theorem hypothesis. 

On the other hand, if no such $v_1$-trajectory $\g$ exists, then we claim the existence of a codimension one closed  submanifold $N \subset \d_1X$, which separates $\d_1X$ in two manifolds, $A \supset \Sigma_1^+$ and $B \supset \Sigma_1^-$ ($\d A = N = \d B$), such that the vector field $v_1$, or rather its perturbation $\tilde v_1$, is \emph{transversal} to $N$ and points \emph{outward} of $A$. Indeed, for each critical point $x \in \Sigma_1^+$, in the local Morse coordinates $(y_1, \dots, y_n)$ on $\d_1X$, consider a small closed $\e$-disk $D_\e^n(x) =\{\sum_k y_k^2 = \e^2 \}$ centered on the critical point $x$.  Denote by $U_\e(x) \subset \d_1X$ the closure of the union of \emph{downward}  trajectories of the $v_1$-flow passing through the points of $D_\e^n(x)$ (see Fig. 6, the left diagram). Let $A_\e$ be the union $\cup_{x \in \Sigma_1^+}\,  U_\e(x)$ (see Fig. 6, the right diagram). 

Since we assume that no descending  $v_1$-trajectory $\g$ links a point of  $\Sigma_1^+$ to a point of $\Sigma_1^-$, we can choose the disks $\{D_\e^n(x)\}_{x \in \Sigma_1^+}$ so small that the set $\Sigma_1^-$ belongs to the complement $\d_1X \setminus A_\e$. 

For each $x \in \Sigma_1^+$, the zero cone $\{Hess_x(f|_{\d_1X}) = 0\}$ of the Morse function $f|_{\d_1X}$ separates the sphere $\d D_\e^n(x)$ into two handles, $H_\e^-(x)$ and $H_\e^+(x)$ (each being a product of a sphere with a disk). We denote by $H^-(x)$ the handle in $\d D_\e^n(x)$ whose spherical core is formed by the intersection of the unstable disk through $x$ with the sphere $\d D_\e^n(x)$. Then, by definition, the set $U_\e(x)$ is a collection of downward trajectories through the points of $H^-(x)$ union with $D_\e^n(x)$. Note that the downward trajectories from a different set $U_\e(y)$ could enter the disk $D_\e^n(x)$ only through the complementary handle $H^+(x) := \d D_\e^n(x) \setminus H^-(x)$ in its boundary. As a result, $U_\e(x) \cup U_\e(y)$ is a manifold whose piecewise smooth boundary could have corners (see Fig. 6, the right diagram)  Similarly, $A_\e$ is a domain in $\d_1X$ whose boundary is  piecewise smooth manifold with corners. 

\begin{figure}[ht]\label{fig.5}
\centerline{\includegraphics[height=4in,width=5.5in]{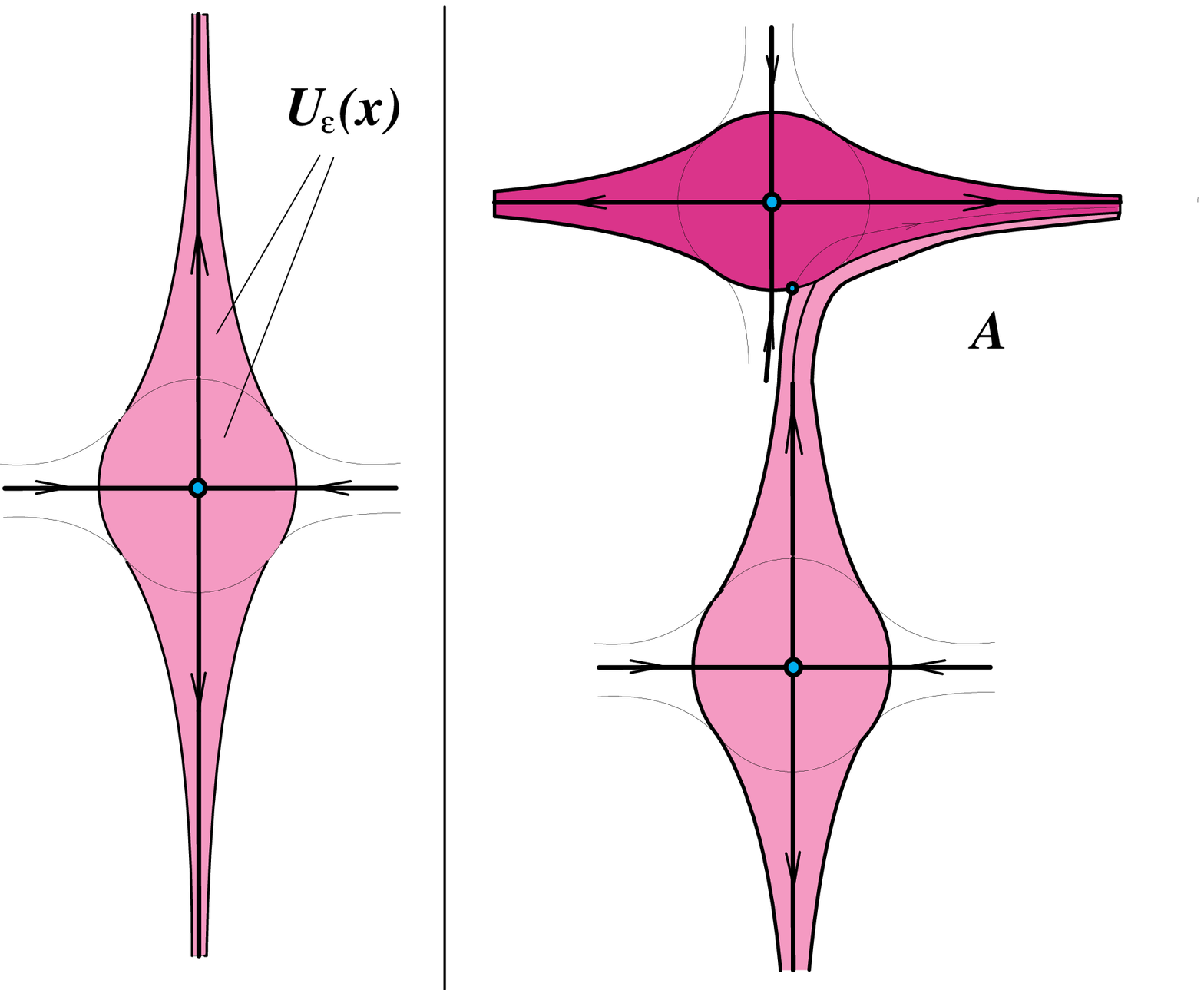}} 
\bigskip
\caption{}
\end{figure}


Since $A_\e$ consists of the downward trajectories of $v_1$,  if $x \in A_\e$, then any point $y \in \g_x$ which can be reached from $x$ following the field $-v_1$ (for short, ``is below $x$") belongs to $A_\e$ as well. Therefore the boundary $\d A_\e$ is assembled either from downward trajectories or from singletons; the singletons are contributed by some portions of $\coprod_{x \in \Sigma_1^+}\d D^n_\e(x)$ where $v_1$ points outside of the relevant disk $D^n_\e(x)$. Thus either $v_1$  is tangent to $\d A_\e$, or it points outside $A_\e$. 

Away from $\Sigma_1^+ \cup \Sigma_1^-$, $v_1 \neq 0$ is of the $f$-gradient type. Thus, in each tangent space $T_x(\d_1X)$, where $x \in \d_1X \setminus \Sigma_1$, there is an open cone $\mathcal C_x(f)$ comprised of $f$-gradient type vectors, and $v_1(x) \in \mathcal C_x(f)$. Therefore, in the vicinity of $\d A_\e$, we can perturb $v_1$  to a new field $\tilde v_1$ of the $f$-gradient type, so that  $\tilde v_1$ points strictly outside $A_\e$ and still $\tilde v_1(x) \in \mathcal C_x(f)$ for all $x \in \d A_\e$. It is possible to smoothen the boundary $\d A_\e$ so that,  with respect to a new smooth boundary $\d \tilde A_\e$,  the field $\tilde v_1$ still points outside $\tilde A_\e \supset \Sigma_1^+$, the new domain bounded by $\d \tilde A_\e$, and $\tilde v_1(x) \in \mathcal C_x(f)$ for all $x \in \d \tilde A_\e$. 

Note that if $f(\Sigma_1^+) < c < f(\Sigma_1^-)$, then $N := f^{-1}(c)$ can serve as a separator.

Let $\tilde A := \tilde A_\e$ and $\tilde B := \overline{\d_1X \setminus \tilde A}$. With the separator $N = \d \tilde A$ in place, consider a smooth function $\tilde h: \d_1X \rightarrow \R$ with the properties: 
\begin{enumerate} 
\item zero is a regular value of $\tilde h$, and  $\tilde h^{-1}(0) = N$,
\item $\tilde h^{-1}((-\infty, 0]) = \tilde A$, $\tilde h^{-1}([0, +\infty) = \tilde B$, 
\item $\tilde h = h$ in a neighborhood of $\Sigma_1^+ \cup \Sigma_1^-$,
\item  $\tilde v := \tilde v_1 + \tilde h\cdot n \in \mathcal C(f)$, where $n$ is the inward normal to $\d_1X$ in $X$.
\end{enumerate}  
Note that the field $\tilde v$ points inside of $X$ along $\tilde A$ and outside of $X$ along $\tilde B$. It also points outside of $\tilde A$ along $N = \tilde A \cap \tilde B$. As a result, we conclude that $\d_2^-X(\tilde v) = N$ and  $\d_2^+X(\tilde v) = \emptyset$; in other words, $\tilde v$ is boundary convex. Note that $\tilde v_1$ can be chosen arbitrary close to $v_1$. Ineeded, employing Theorem \ref{th3.2}, we can perturb $\tilde v_1$ to insure its genericity with respect to the pair $(\d_1^+X(\tilde v), \d_2X(\tilde v))$, and thus the boundary genericity of $\tilde v$ itself. 
\smallskip 

The argument in the concave case, which deals with the relation between the property $\d_2^- X(v) = \emptyset$ and the absence of an ascending $v_1$-trajectory $\g: \R \to \d _1X$,  connecting  $\Sigma_1^+ $ to $\Sigma_1^-$, is analogous. We just need to switch the polarities of the relevant sets. 
\end{proof}

Now we need to introduce a number basic notions to which we will return on many occasions in the future. 

\begin{defn}\label{def4.2} Let $\omega$ be a differential $1$-form  on a manifold $Y$.  

We say that a path $\g: [0, 1] \to Y$ is $\omega$-\emph{positive} ($\omega$-\emph{negative}), if , $\omega(\dot \g(t)) > 0 \; (< 0)$ for all  values of the parameter $t \in (0, 1)$.  \hfill \qed
\end{defn}


\begin{defn}\label{def4.3}  Let $\omega$ be a closed differential $1$-form  on a manifold $Y$, equipped with a Riemannian metric $g$. We say that a vector field $v$ on $Y$ is the \emph{gradient} of $\omega$ (and denote it ``$\nabla_g\omega$"), if $\omega(w) = \langle v, w \rangle_g$ for any vector field $w$ on $Y$. \hfill\qed
\end{defn}

\begin{defn}\label{def4.4}  Let $\omega$ be a differential $1$-form  on a manifold $Y$ and let $\Sigma_\omega$ be the set of points $y \in Y$, where $\omega: T_yY \to \R$ is the zero map.  Assume that $\omega = df$ for some smooth function $f$ in the vicinity of $\Sigma_\omega$. 

We say that a vector field $v$ is of $\omega$-\emph{gradient type} if $\omega(v) > 0$ on $Y \setminus \Sigma_\omega$ and $v = \nabla_gf$ in the vicinity of $\Sigma_\omega$. Here $g$ is \emph{some} Riemannian metric in the vicinity of $\Sigma_\omega$ (cf. Definition \ref{def2.2}). \hfill \qed
\end{defn}

We are in position to formulate a generalization of Theorem \ref{th4.1} for closed differential $1$-forms---another instance of somewhat weaker ``holographic phenomenon", this time for  fields which may not be gradient-like globally.

\begin{thm}\label{th4.2} Let $\omega$ be a closed $1$-form on a compact manifold $X$, equipped with a Riemannian metric $g$.  Assume that $\omega$ and $\omega|_{\d_1X}$ have only  Morse-type singularities. Let  the gradient $v: = \nabla_g \omega$ be a boundary generic field, and let $v_1:= \nabla_{g|_{\d_1X}}(\omega|_{\d_1X})$.  

If $\d_2^\pm X(v) = \emptyset$, then 
there is no $\omega$-ascending $v_1$-trajectory $\g \subset \d _1X$, such that 
$$\lim_{t \rightarrow -\infty} \g(t) \in \Sigma_1^\mp \;\; \text{and} \;\;  \lim_{t \rightarrow +\infty} \g(t) \in \Sigma_1^\pm.$$

Assume that there exists a codimension one submanifold $N \subset \d_1X$, which separates $\Sigma_1^+$ and $\Sigma_1^-$ and such that the field $v_1$ is  transversal to $N$ and points outwards/inwards of the domain in $\d_1X$ that is bounded by $N$ and contains $\Sigma_1^+$.  Then one can deform the $\omega$-gradient vector fields $(v, v_1)$ to a new boundary generic pair $(\tilde v, \tilde v_1)$ of the $\omega$-gradient type so that $\d_2^\pm X(\tilde v) = \emptyset$.   
\end{thm}

\begin{proof} The $(\omega|_{\d_1X})$-gradient fields $v_1$ on $\d_1X$ are characterized by the property $\omega(v_1) > 0$, valid on the locus where $\omega|_{\d_1X} \neq 0$. Usually, in this setting, we do not have a natural choice for the wall $N \subset \d_1X$ which would separate the singularities of opposite polarities $\Sigma_1^+ = \Sigma_1^+(\omega)$ and $\Sigma_1^- = \Sigma_1^-(\omega)$ and such that the field $v_1$  would be  transversal to $N$. 
It seems unlikely that the absence of an ascending $v_1$-trajectory which links $\Sigma_1^-$ with $\Sigma_1^+$ is sufficient to guarantee the existence of a separator $N$. However, in the presence of such separator $N$, the arguments are identical with the ones employed in the proof of Theorem \ref{th4.1}. 
\end{proof}

{\bf Remark 4.1.} In Theorem \ref{th4.1}  and Theorem \ref{th4.2},  the partition $\Sigma_1^+ \coprod \Sigma_1^-$ of the singular set $\Sigma_1$ must satisfy some basic relations:
$$\sum_{x \in \Sigma_1^+} \mathsf{ind}_x(v_1) + \sum_{x \in \Sigma_1^-} \mathsf{ind}_x(v_1) = 0, \; \text{when}\; n+1 \equiv 0\; mod\; 2,$$
$$\sum_{x \in \Sigma_1^+} \mathsf{ind}_x(v_1) + \sum_{x \in \Sigma_1^-} \mathsf{ind}_x(v_1) = 2\cdot \chi(X), \; \text{when}\; n+1 \equiv 1\; mod\; 2.$$
These relations reflect the fact that $\chi(\d_1X) = 0$ when $n+1 \equiv 0\; mod\; 2$,  and $\chi(\d_1X) = 2\cdot \chi(X)$ when $n+1 \equiv 1\; mod\; 2$.
\hfill\qed
\smallskip

Given a metric $g$ on a Riemannian $(n+1)$-manifold $X$, let us recall a definition of the Hodge  Star Operator $\ast_g: T^\ast(X) \to \bigwedge^n T^\ast(X)$.  

Pick a local basis $\a := (\a_1, \dots , \a_{n+1})$ of $1$-forms in $T^\ast(X)$ and consider the associated basis $$\a^\vee := (\dots ,\;\; (-1)^{k +1}\a_1 \wedge \dots \, ^{\vee^{(k)}}  \dots \wedge \a_{n+1},\;\; \dots )$$ of $\bigwedge^n T^\ast(X)$, where $1 \leq k \leq n+1$ and the symbol ``$\vee^{(k)} $" stands for omitting the $k$-th form $\a_k$ from the product $\a_1 \wedge \dots  \wedge \a_{n+1}$. 

Assume that, in the dual to $\a$  basis $\a^\ast$ of $T(X)$, the metric $g$ is locally given by a matrix $\mathsf g  = (g_{ik})$. 
Then 
the matrix $\mathsf G$ of the $\ast_g$-operator in the bases $\a$, $\a^\vee$ is given by the formula
\begin{eqnarray}\label{eq4.1} 
\mathsf G = \sqrt{\det(\mathsf g)} \cdot \mathsf g^{-1}
\end{eqnarray}
, whence $\det(\mathsf G) = (\det(\mathsf g))^{\frac{n -1}{2}}$.
\smallskip

\begin{defn}\label{def4.5} A closed differential $1$-form $\omega$ on a compact manifold $Y$ is called \emph{intrinsically harmonic} if there exists a Riemannian metric  $g$ on $Y$ such that the form $\ast_g(\omega)$ is \emph{closed}. \hfill \qed
\end{defn}

\noindent {\bf Example 4.3.}
Let $Y$ be a closed 
smooth manifold and $H: Y \to S^1$ a smooth map with isolated Morse-type singularities. Consider the closed $1$-form $\omega := H^\ast(d\theta)$,  the pull-back of the canonic $1$-form $d\theta$ on the circle $S^1$. Assume that one of the $H$-fibers, $F_0 := H^{-1}(\ast)$,  is connected. Then $\omega$ is intrinsically harmonic \cite{FKL}. 
\hfill\qed

Let $\Sigma_\omega$ denote the singularity set of a closed $1$-form $\omega$ on a compact manifold $Y$. We assume that  $\Sigma_\omega \subset \textup{int}(Y)$. 

By  Calabi's Proposition 1 \cite{Ca}, $\omega$ is intrinsically harmonic if and only if  through every point $y \in Y \setminus  \Sigma_\omega$ there is a $\omega$-\emph{positive} path $\g$ which either is a loop, or which starts and terminates at the boundary $\d Y$. 

\begin{thm}\label{th4.3} Let  $\omega$ be a closed $1$-form on a Riemannian manifold $X$, such that $\Sigma_\omega \subset \textup{int}(X)$.  Assume that $\omega|_{\d_1X}$, the restriction of $\omega$ to $T(\d_1X)$, is a harmonic form\footnote{This assumption implies that $H^1(\d_1X; \R) \neq 0$, provided $\omega_{|_{\d_1X}} \neq 0$.}. 

Then the gradient field $v := \nabla\omega$ is not boundary convex or  boundary concave (that is, $\d^+_2X(v) \neq \emptyset$ \emph{and} $\d^-_2X(v) \neq \emptyset$). Thus, if $\d_2X(v)$ is connected, then $\d_3X(v) \neq \emptyset$.
\end{thm}

\begin{proof} We abbreviate $\d_j^\pm X(v)$ to $\d_j^\pm X$ and $\ast_{g|_{\d_1X}}$ to $\ast_\d$. Here $\ast_{\d}$ is the $\ast$-operator on the boundary of $X$ with respect to the given Riemannian metric $g$ on $X$. 

If $\ast_{\d}(\omega|_{\d_1X})$ is a closed $(n -1)$-form on $\d_1X$, then by the Stokes Theorem, $$\int_{\d_2X} \ast_{\d}(\omega|_{\d_1X}) =  \int_{\d_1X^+} d(\ast_{\d}(\omega|_{\d_1X})) = 0.$$ However, for a concave/convex gradient field $v =\nabla\omega$, the $(n-1)$-form $\ast_{\d}(\omega|_{\d_1X})$, being restricted to $\d_2X$, is proportional  to the volume form of $\d_2X$ with negative/positive functional coefficient.  Indeed, at the points of $\d_2^+X$, the angle between $v$ and the normal $n$ to $\d_2X$ in $\d_1^+X$ is acute, while it is obtuse at the points of $\d_2^-X$. Therefore, $\int_{\d_2X} \ast_{\d}(\omega|_{\d_1X}) \neq 0$ when either $\d_2^+X = \emptyset$ or $\d_2^-X = \emptyset$. The resulting contradiction proves that  $\d_2^+X \neq \emptyset$ and $\d_2^-X \neq \emptyset$. 

Therefore, when $\d_2X$ is connected, then $\d_2^+X$ and $\d_2^-X$ must share the common nonempty boundary $\d_3X$---the gradient field $v$ must have cuspidal points.
\end{proof}
\smallskip

\noindent {\bf Example 4.4.} Let $X$ be a compact  smooth manifold and $H: X \to S^1$ a smooth map with isolated Morse-type singularities. Consider the closed $1$-form $\omega := H^\ast(d\theta)$,  the pull-back of the canonic $1$-form $d\theta$ on the circle $S^1$. Assume that one of the fibers of the map $H: \d_1X \to S^1$ is connected. Then there exists a metric $g$ on $X$ such that the form $\omega^\d := H^\ast(d\theta)|_{\d_1X}$ is harmonic  (\cite{Ca}, \cite{FKL}). Consider the gradient field $v := \nabla_g(\omega)$. Then by Theorem \ref{th4.3}, $\d_2^+X(v) \neq \emptyset$ and $\d_2^-X(v) \neq \emptyset$ for any metric $g$ that ``harmonizes" $\omega^\d$. \hfill\qed

\begin{defn}\label{def4.6} A non-vanishing vector field $v$ on a compact manifold $X$ is called \emph{traversing} if each $v$-trajectory is either a closed segment or a singleton which belongs to $\d X$. \hfill\qed
\end{defn}

\noindent {\bf Remark 4.2.} The definition excludes fields with zeros in $X$ (they will generate trajectories that are homeomorphic to open or semi-open intervals) and fields with closed trajectories. Note that all gradient-like fields of nonsingular functions are traversing, but the gradient-like fields of nonsingular closed $1$-forms may not be traversing! \hfill\qed

\begin{lem}\label{lem4.1} Any traversing vector field is of the gradient type.
\end{lem}
\begin{proof} Let $v$ be a traversing field on $X$. We extend the pair $(X, v)$ to a pair $(\hat X, \hat v)$ so that $X$ is properly contained in $\hat X$ and $\hat v \neq 0$. 

Then every $v$-trajectory $\g^\star \subset X$ has a local transversal compact section $S_{\g^\star} \subset \textup{int}(\hat X)$ of the $\hat v$-flow. We can choose $S_{\g^\star}$ to be diffeomorphic to a $n$-dimensional ball with its center at the singleton $\g^\star \cap S_{\g^\star}$. We denote by $\tilde U_{\g^\star}$ the union of $\hat v$-trajectories through $S_{\g^\star}$. 

For each $v$-trajectory $\g^\star$, there exists a section $S_{\g^\star}$ so that the set $\tilde U_{\g^\star}$ contains a compact cylinder $\hat U_{\g^\star} \approx S_{\g^\star} \times [-a_{\g^\star}, b_{\g^\star}]$, where $a_{\g^\star}, b_{\g^\star}$ are positive constants (which depend on $\g^\star$), with the properties: 
\begin{enumerate}
\item $\hat U_{\g^\star} \supset \tilde U_{\g^\star} \cap X$,
\item for any $\hat v$-trajectory $\tilde\g$ through $S_{\g^\star}$, the intersection $\hat\g := \tilde\g \cap \hat U_\g$ is a segment, 
\item the point $\tilde\g \cap S_{\g^\star}$ belongs to the interior of segment $\hat\g$.
\end{enumerate}

Then the collection $\mathcal U := \{\hat U_{\g^\star} \cap X\}_{\g^\star}$ forms a cover of $X$. Since $X \subset \mathsf{int}(\hat X)$ is compact,  we can choose a finite subcover $\mathcal U' \subset \mathcal U$ of $X$. 

For each $\hat U_{\g^\star} \cap X \in \mathcal U'$ and the corresponding section $S_{\g^\star}$, we produce a smooth function $\phi_{\g^\star}: \hat U_{\g^\star} \to \R$  by integrating the vector field $\hat v$ and using $S_{\g^\star}$ as the initial location for the integration. More accurately, let $$\psi_{\hat\g}: [-a_{\g^\star}, b_{\g^\star}]  \to \hat U_{\g^\star}$$ be the parametrization of a typical trajectory $\hat\g \subset \hat U_{\g^\star}$, such that  $$\frac{d}{dt}\psi_{\hat\g}(\tau) = \hat v(\psi_{\hat\g}(\tau))$$ for all $\tau \in [-a_{\g^\star}, b_{\g^\star}]$ and $\psi_{\hat\g}(0) = \hat\g \cap S_{\g^\star}$. This bijective parametrization introduces a smooth product structure $$\Phi: \hat U_{\g^\star} \approx S_{\g^\star} \times [-a_{\g^\star}, b_{\g^\star}]$$ by the formula $\Phi(x) := \big(\hat\g_x \cap S_{\g^\star}, \, (\psi_{\hat\g_x}^{-1}(x)\big)$.

We define a smooth function $\hat\phi_{\g^\star}: \hat U_{\g^\star} \to \R$ by the formula $x \to \psi_{\hat\g_x}^{-1}(x)$ and denote it (quite appropriately) by the symbol $\int_{S_{\g^\star}}^x \hat v$.  

Let $\chi_{\g^\star}: S_{\g^\star} \to \R_+$ be a smooth non-negative function that vanishes only on the boundary $\d S_{\g^\star}$. Let $\tilde\chi_{\g^\star}: \hat U_{\g^\star} \to \R_+$ denote the composition of the $\hat v$-directed projection $\pi_{\g^\star}: \hat U_{\g^\star} \to S_{\g^\star}$ with the function $\chi_{\g^\star}$. Since $\tilde\chi_{\g^\star}$ vanishes on $\d \hat U_{\g^\star} \cap X$, the function extends smoothly on $X$ to produce a smooth function $\hat\chi_{\g^\star}: X \to \R_+$ with the support in $ \hat U_{\g^\star} \cap X$.

Now consider the smooth function
\begin{eqnarray}\label{eq4.2}  
f(x) := \sum_{\hat U_{\g^\star} \in \, \mathcal U'} \hat\chi_{\g^\star}(x) \Big(\int_{S_{\g^\star}}^x \hat v\Big)
\end{eqnarray} 
 
It is  well-defined on $X$. Let us compute its $v$-directional derivative:
\begin{eqnarray}\label{eq4.3}  
\mathcal L_v(f(x)) = \sum_{\hat U_{\g^\star} \in\, \mathcal U'} \mathcal L_v\Big[\hat\chi_{\g^\star}(x)\cdot \Big(\int_{S_{\g^\star}}^x \hat v\Big)\Big] \nonumber \\
=  \sum_{\hat U_{\g^\star} \in \, \mathcal U'} \hat\chi_{\g^\star}(x)\, \mathcal L_v\Big(\int_{S_{\g^\star}}^x \hat v\Big) > 0 
\end{eqnarray}

Let us explain formula \ref{eq4.3}. By the very definition of $\hat\chi_{\g^\star}$, it is constant on each $\hat v$-trajectory, so that $\mathcal L_v(\hat\chi_{\g^\star}) = 0$. Also, $\hat\chi_{\g^\star} > 0$ in $\textup{int}(\hat U_{\g^\star})$. At the same time, $\mathcal L_v\big(\int_{S_{\g^\star}}^x \hat v\big) > 0$, since $\frac{d}{dt}\psi_{\hat\g} = \hat v(\psi_{\hat\g}) \neq 0$ and $\psi_{\hat\g}$ increases in the direction of $v$. 
Finally, each $x \in X$ belongs to the interior of some set $\hat U_{\g^\star}$.

Therefore, $df(v) = \mathcal L_v(f) > 0$, so that $v$ is a gradient-like field for $v$.
\end{proof}

\begin{cor}\label{cor4.1} Let $X$ be a smooth compact manifold with boundary. Then $\mathcal V_{\mathsf{trav}}(X)$---the space of traversing vector fields on $X$---is nonempty and  coincides with  the intersection $\mathcal V_{\mathsf{grad}}(X) \cap \mathcal V_{\neq 0}(X)$, where $\mathcal V_{\mathsf{grad}}(X)$ denotes the space of gradient-like fields, and $\mathcal V_{\neq 0}(X)$ the space of all non-vanishing fields on $X$.
\end{cor}

\begin{proof} By definition, any traversing field $v$ on $X$ does not vanish. By Lemma \ref{lem4.1}, $v$ must be of the gradient type. Thus $$\mathcal V_{\mathsf{trav}}(X) \subset \mathcal V_{\mathsf{grad}}(X) \cap \mathcal V_{\neq 0}(X).$$

On the other hand, for a compact $X$ with a gradient-like $v \neq 0$, each $v$-trajectory $\g_x$ through $x \in \textup{int}(X)$ must reach the boundary in both finite positive and negative times (since it is controlled by some Lyapunov function $f$). 

As a result, $$\mathcal V_{\mathsf{trav}}(X) = \mathcal V_{\mathsf{grad}}(X) \cap \mathcal V_{\neq 0}(X).$$ 
\smallskip

It remains to show that $\mathcal V_{\mathsf{trav}}(X) \neq \emptyset$. By Lemma \ref{lem3.2}, $\mathcal V_{\mathsf{grad}}(X) \cap \mathcal V_{\neq 0}(X) \neq \emptyset$, which implies that $\mathcal V_{\mathsf{trav}}(X) \neq \emptyset$.
\end{proof}






\begin{figure}[ht]
\centerline{\includegraphics[height=1.7in,width=3in]{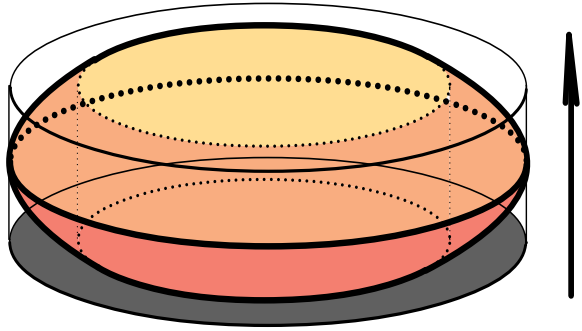}} 
\bigskip
\caption{The existence of a traversing boundary convex field $v$ (the constant vertical field) on a $(n+1)$-manifold $X$ (the ellipsoid-bounded solid) implies that topologically it is a product of a compact $n$-manifold $Y$ (the elliptical shadow) with an interval.}
\end{figure}

There are simple topological obstructions to boundary convexity of \emph{any} gradient-like \emph{nonvanishing} field on a given manifold $X$. The next lemma testifies that the existence of boundary convex traversing  fields $v$ imposes severe restrictions on the topology of the manifold $X$.  

\begin{lem}\label{lem4.2} A connected $(n +1)$-manifold $X$ admits a \emph{boundary convex}  traversing\footnote{equivalently, a non-vanishing gradient-like field} field $v$, if and only if, $X$ is diffeomorphic to a product of a connected compact $n$-manifold and a segment, the corners of the product being smoothly rounded. 
\end{lem}

\begin{proof} Indeed, if such convex $v$ exists, $\d_1^+X := \d_1^+X(v)$ must be a deformation retract of $X$: just use the down flow to produce the retraction. Therefore, when $\d_2^+X(v) = \emptyset$, then $X$ is homeomorphic to the quotient space $\{(\d_1^+X) \times [0, 1]\}/\sim$, where the equivalence relation "$\sim$" is defined by collapsing each segment $\{x\times [0,1]\}_{x \in \d_2^-X}$ to a point. If we round the corners generated in the collapse, we will get a diffeomorphism between $X$ and the ``lens"  $\{(\d_1^+X) \times [0, 1]\}/\sim$ (see Fig. 5). 

On the other hand, any product $Y \times [0, 1]$, whose  conners $\d Y \times \d([0, 1])$ being rounded, admits a field of the desired boundary convex type. 
\end{proof}

\begin{cor}\label{cor4.2} For all $n \neq 4$, any  smooth compact contractible $(n +1)$-manifold $X$, which admits a boundary convex traversing field, is diffeomorphic to the standard $(n +1)$-disk.
\end{cor}

\begin{proof}  By Lemma \ref{lem4.2}, $X$ is diffeomorphic to a product of a fake $n$-disk $Y$ with $[0, 1]$, the corners of the product being rounded. 

For $n = 3$, by Perelman's results \cite{P1}, \cite{P2}, $Y$ is diffeomorphic to the standard $3$-disk. Thus $X$ is diffeomorphic to the standart $4$-disk.

For $n = 4$, we do not know whether $Y$ is a standard $4$-disk.

For $n \geq 5$, the $h$-cobordism theorem \cite{Sm} implies that any fake $n$-disk is diffeomorphic to the standard disk.

This leaves only the case of $5$-dimensional $X$ wide open.
\end{proof}

We notice that $H_n(X; \Z) \neq 0$ is an \emph{obstruction} to finding \emph{boundary convex} traversing $v$ on a $(n +1)$-dimensional manifold $X$ with a connected boundary. 

\begin{cor}\label{cor4.3} Let $X$ be a smooth connected compact $(n+1)$-manifold with boundary, which admits a boundary \emph{convex} traversing field. 

If $H_n(X; \Z) \neq 0$, then $X$ is diffeomorphic to the product $Y \times [0, 1]$, where $Y$ is a \emph{closed} manifold.

In particular, no connected $X$ with boundary $\d X \neq \emptyset$, whose number of connected components differs from two, and with the property $H_n(X; \Z) \neq 0$ admits a boundary convex traversing field. 
\end{cor}

\begin{proof} If such boundary convex traversing field $v$ exists, $\d_1^+X$ must be a deformation retract of $X$. Therefore, for a connected $X$, $\d_1^+X$ must be connected as well. 

On the other hand, if  $\d(\d_1^+X) = \d_2X \neq \emptyset$, then the connected $\d_1^+X$ must be of a homotopy type of a $(n-1)$-dimensional complex. In such a case, the groups $H_n(\d_1^+X; \Z) \approx H_n(X; \Z)$ must vanish. 

Thus when $H_n(X; \Z) \neq 0$ and $v$ is boundary convex, the only remaining option is $\d_2X = \emptyset$, which implies that $\d(\d_1^+X) =  \emptyset$---the manifold $\d_1^+X$ is closed. In such a case, $X$ is a product of a connected closed $n$-manifold with an interval; so the boundary $\d X$ must be the union of two diffeomorphic components.
\end{proof}

As with the boundary convex traversing fields, perhaps, there are topological obstructions to the existence of a boundary \emph{concave} traversing field on a given manifold? At the present time, the contours of the universe of such obstructions are murky. We know only that the disk $D^2$ does not admit a non-vanishing boundary concave field (see Example 4.4). 

\begin{lem}\label{lem4.3} If a boundary generic vector field $v$ on an \emph{even}-dimensional compact orientable manifold $X$ is boundary concave, then its index 
$$\mathsf{Ind}(v) = \chi(X) + \chi(\d_1^+ X) = \chi(X) + \frac{1}{2}\cdot \chi(\d_2 X).$$ 

If a boundary generic vector field $v$ on an \emph{odd}-dimensional compact orientable manifold $X$ is boundary concave, then its index $$\mathsf{Ind}(v) = \chi(X) - \chi(\d_1^+X).$$ Thus, for all boundary concave fields $v$ with a fixed value of index $\mathsf{Ind}(v)$,  the Euler number $\chi(\d_1^+ X)$ is a topological invariant. 
\end{lem}

\begin{proof} For a boundary concave field $v$, $\d(\d_1^+X) := \d_2X = \d_2^+X$. Therefore, the Morse formula \ref{eq2.2} reduces to the equation 
\begin{eqnarray}\label{eq4.4} 
\mathsf{Ind}(v) = \chi(X) - \chi(\d_1^+X) +\chi(\d_2X).
\end{eqnarray} 

Recall that,  for any orientable \emph{odd}-dimensional manifold $Y$,  $\chi(\d Y) =  2\cdot \chi(Y)$. Therefore, when $\dim(X) \equiv 0 \; mod \; 2$, we get   $2\cdot\chi(\d_1^+X) =  \chi(\d_2X)$. Thus formula \ref{eq4.4} transforms into $$\mathsf{Ind}(v) = \chi(X) + \chi(\d_1^+X) = \chi(X) + \frac{1}{2}\cdot \chi(\d_2 X).$$ 

For an odd-dimensional $X$, the closed manifold $\d_2X$ is odd-dimensional, so $\chi(\d_2X) = 0$. Therefore $$\mathsf{Ind}(v) = \chi(X) - \chi(\d_1^+X).$$
\end{proof}







\begin{cor}\label{cor4.4} Let $X$ be a $4$-dimensional oriented smooth and compact manifold with boundary.

If $\chi(X) < 0$, then for any  boundary generic concave vector field $v$ on $X$ of index $0$, the locus $\d_2X = \d_2^+X$ contains at least  $|\chi(X)|$ two-dimensional spheres. 
\end{cor}

\begin{proof} Since $\d_2X$ is a closed orientable $2$-manifold, its Euler number is positive only if $\d_2X$ contains sufficiently many $2$-spheres.  By Lemma \ref{lem4.3}, $\chi(\d_2X) = -2\cdot \chi(X) > 0$. Therefore $\d_2X$ contains at least  $|\chi(X)|$ two-dimensional spheres.
\end{proof}

\smallskip
\noindent {\bf Example 4.4.} Let $X = D^2$, the $2$-dimensional ball. If $v \neq 0$ on $X$, then by the Morse formula, $$1 - \chi(\d_1^+X) +  \chi(\d_2^+X) = 0.$$ If $\d_1^+X$ consists of $k$ arcs, then by this formula, $\#(\d_2^+X) = k-1$. At the same time, $\#(\d_2X) = 2k$. Therefore, $\#(\d_2^-X) = k +1 > 0$. So we conclude that $D^2$ does not admit a non-vanishing field with $\d_2^-X = \emptyset$, that is, a boundary concave field. \smallskip

At the same time, if we delete any number of disjoint open disks  from $D^2$, the remaining surface $X$ admits a concave non-vanishing gradient-like field: indeed, consider the radial field in an annulus $A$ and delete from $A$ any non-negative number of small round disks. The radial field $v$ on $A$, being restricted to $X$, is evidently of the gradient type and concave with respect to $\d X$.

Note that, if a connected compact surface $X$ admits a generic traversing concave field $v$, then $X$ is homeomorphic either to a thickening  of a finite graph $\Gamma$ whose vertexes all have valency $3$, or to an annulus.  
\hfill\qed
 \smallskip
 
In the previous example, we have seen that the disk $D^2$ does not admit a non-vanishing concave field. In contrast, $D^3$ does admit a boundary generic  concave non-vanishing field: just consider the restriction of the Hopf field $v$ on $S^3$ to the northern hemisphere $D^3 \subset S^3$. For the unitary disk $D^3 \subset \R^3$ centered at the origin, informally, we can describe $v$ as the sum of the velocity field of the solid $D^3$, spinning around the $z$-axis, with the solenoidal field of the loop $L := \{x^2 + y^2 = 4/9,\, z = 0\}$.  However,  this field $v$ is not of the traversing type: it has closed trajectories (residing in the solid torus $dist(\sim, L) \leq 1/3$). 
\smallskip

These observations encourage us to formulate

\begin{conj}\label{conj4.1} The standard $(n+ 1)$-disk  $D^{n+1}$ does not admit a  traversing boundary concave vector field. \hfill\qed
\end{conj}

The construction of a boundary concave field on a $2$-disk with holes (see Example 4.4) admits a simple generalization.
\smallskip

\noindent {\bf Example 4.5.}
Consider a closed $n$-manifold $Y$. Let $\{Z_i \subset Y\}_{1 \leq i \leq s}$ be compact submanifolds also of dimension $n$. Let $W := Y \times [0, 1]$. We pick $s$ disjointed close intervals $\{I_i\}_i$ in the interval $[0, 1]$. Then we form the product $U_i := Z_i \times I_i$. By rounding the corners of $U_i$, we get a $(n+1)$-manifold $V_i \subset U_i$ so that each segment $z \times I_i$, where $z \in  \textup{Int}(Z_i)$, hits $V_i$ along a closed segment, and each segment $z \times I_i$, where $z \in  \d(Z_i)$, hits $V_i$ along a singleton.  

Form the manifold $X := W \setminus \coprod_i V_i$. Its boundary consists of two copies of $Y$ together with the disjoint union of $\d V_i$ (they are the doubles of $Z_i$'s). The obvious vertical field $v$ on $W$, being restricted to $X$, is boundary concave. In fact, $\d_1^+X(v) = Y \times \{0\} \coprod A$, where $A \approx \coprod_i Z_i$, and  $\d_2^+X(v) \approx \coprod_i \d Z_i$. 
\hfill\qed
\smallskip

These examples lead to few interesting questions:
\smallskip

\noindent {\bf Question 4.1.} Which compact manifolds admit \emph{boundary concave} non-vanishing vector fields? Which compact manifolds admit \emph{boundary concave} non-vanishing \emph{gradient-like} fields? \hfill\qed
\smallskip

Despite the ``natural" flavor of these questions, we have a limited understanding of the general answers. Nevertheless, feeling a bit adventurous, let us state briefly what kind of answer one might anticipate. This anticipation is based on a better understanding of boundary concave traversing fields on $3$-folds (see \cite{BP}, \cite{K}).

We conjecture that an $(n+1)$-dimensional $X$ admits a traversing concave field $v$ such that $\d_2X(v) = \d_2^+X(v) \neq \emptyset$ if (perhaps, if and only if) $X$ has a ``special trivalent" simple $n$-dimensional spine $K \subset T_X$, where $T_X$ denotes a smooth triangulation of $X$ (see \cite{Ma} for the definitions of simple spines and for the description of their local topology). Here ``special trivalent" means that each $(n-1)$-simplex from the singular set $SK$ of $K$ is adjacent to exactly \emph{three} $n$-simplexes from $K$. Moreover, the vicinity of $SK$ in $K$ admits an \emph{oriented branching} as in \cite{BP}. 
\smallskip   

When the $(n+1)$-manifold in question is specially manufactured from a closed $(n+1)$-manifold by removing a number of $(n+1)$-disks, another paper from this series  will provide us with a wast gallery of manifolds which admit traversing concave fields. 

\smallskip

\section{Morse Stratifications of the Boundary $3$-convex and $3$-concave Fields}

We have seen that the boundary $2$-convexity of traversing fields on $X$ has strong implications for the topology of $X$ (for example, see Lemmas \ref{lem4.2}-\ref{lem4.3}, and Corollaries  \ref{cor4.2}-\ref{cor4.4}).

By itself, the boundary 3-convexity and $3$-concavity of traversing fields has no topological significance for the topology of $3$-folds: we have proved in Theorem 9.5 from \cite{K} that, for every 3-fold $X$,  any boundary generic $v$ of the gradient type can be deformed into new such field $\tilde v$ with $\d_3 X(\tilde v) = \emptyset$. However, in conjunction with certain topological constraints on $\d_1^+X$ (like being connected), the boundary 3-convexity has topological implications (see \cite{K}, Corollary 2.3 and Corollary 2.5). 
\smallskip

These observations  suggest two general questions: 
\smallskip

\noindent {\bf Question 5.1.} 
\begin{itemize}
\item Given a manifold $X$, which patterns of  the stratifications $\{\d_j^+X(v) \subset \d_j X(v)\}_j$ are realizable by boundary generic \emph{traversing} fields $v$ on $X$?\footnote{Theorem \ref{th5.1} and Corollary \ref{cor5.1} below gives just a taste of a possible answer.}
 \item Given two such fields, $v_0$ and $v_1$, can we find a linking path $\{v_t\}_{t \in [0, 1]}$ in the space $\mathcal V_{\mathsf{trav}}(X)$ that avoids certain types of singularities?\footnote{When $\dim(X) = 3$,  Theorem 9.5 in \cite{K} addresses some of these questions.} Specifically, if  for some $j > 0$, $\d_j X(v_0) = \emptyset = \d_j X(v_1)$, is there a linking path so that $\d_{j +1}X(v_t) = \emptyset$ for all $t \in [0, 1]$?
\end{itemize}
\smallskip

\noindent {\bf Remark 5.1.} The property of the field $v$ in Question 5.1 being traversing (equivalently, boundary generic and of the gradient type) is the essence of the question. For just boundary generic fields, there are no known restrictions on the patterns of  $\{\d_j^+X(v) \subset \d_j X(v)\}_j$. 

Let us illustrate this remark for the fields $v$ such that $\d_3X(v) = \emptyset$. We divide the boundary  $\d_1X$ into two complementary domains, $Y^+$ and $Y^-$, which share a common boundary $\d Y^+ = \d Y^-$---a closed manifold of dimension $n-1$. It may have several connected components. Next, we divide the manifold $\d Y^+$ into two complementary closed manifolds $Z^+$ and $Z^-$. 

We claim that it is possible to find a boundary generic field $v$ with the properties: $\d_1^\pm X(v) = Y^\pm$,  $\d_2^\pm X(v) = Z^\pm$, and $\d_3X(v) = \emptyset$. The construction of such $v$ is quite familiar (see the arguments in Theorem \ref{th3.2}).

We start with a field $\nu_1$ which is normal to $\d Y^+$ and points outside of $Y^+$ along $Z^-$ and inside of $Y^+$ along $Z^+$. We extend $\nu_1$ to a field $v_1$ tangent to the boundary $\d_1X$ so that $v_1$ has only isolated zeros. Let $\nu$ be the outward normal field of $\d_1X$ in $X$ and $h: \d_1X \to \R$ a smooth function such that $0$ is its regular value and $$h^{-1}((-\infty, 0]) = Y^-,\quad  h^{-1}([0, +\infty)) = Y^+.$$ Along $\d_1X$, form the field $v' = v_1 + h\cdot \nu$ and extend it to a field $v$ on $X$ with isolated singularities in $\textup{int}(X)$. By its construction, $v$ has all the desired properties. Note that here we do not insist on the property $v \neq 0$. \hfill\qed

\smallskip
 
In our inquiry, we are inspired by the Eliashberg surgery theory of folding maps \cite{E1}, \cite{E2}.  In many cases, Eliashberg's results give criteria for realizing given patterns of $\d_2^\pm X \subset \d_1^\pm X$, provided that $\d_3 X = \emptyset$, thus answering Question 5.1. Let us state one such result, Theorem 5.3 from \cite{E2}. 

\begin{thm}[{\bf Eliashberg}]\label{th5.1}  Let $X \subset \R^{n+1}$, $n \geq 2$, be a compact connected smooth submanifold of dimension $(n+1)$. Consider two disjoint  closed and nonempty  $(n-1)$-submanifolds $Z^+$ and $Z^-$ of $\d X$ whose union separates $\d X$ into two complementary $n$-manifolds, $Y^+$ and $Y^-$. Let $\nu$ be the outward normal field of $\d X$ in $X$, and denote by $\deg(\nu)$ the degree of the Gauss map $G_\nu: \d X \to S^n$.  Let $\pi : \R^{n + 1} \to \R^n$ be a linear surjection.

Then the  topological constraints 
\begin{itemize}
\item $\chi(Z^+) - \chi(Z^-) = 0$, \;\;\;\;\;\qquad when $n \equiv 0\; \mod \;2$
\item $\chi(Z^+) - \chi(Z^-) = 2\cdot \deg(\nu)$,  when $n \equiv 1 \; \mod \;2$
\end{itemize} 
are necessary and sufficient for the existence of an orientation-preserving diffeomorphism $h: \R^{n + 1} \to \R^{n + 1}$ with the following properties: 
\begin{itemize}
\item $Z^+ \cup Z^-$ is the fold locus of the map $(\pi \circ h): \d X \to \R^n$, 
\item $\pi \circ h$, being restricted to $Z^+ \cup Z^-$, is a immersion, and the image $(\pi \circ h)(Z^+ \cup Z^-)$ has only transversal self-intersections in $\R^n$,
\item the differential $D(\pi \circ h)$ takes the normal field $\nu|_{Z^+}$ to the field inward  normal to $(\pi\circ h)(Z^+)$ in  $(\pi\circ h)(\d X)$, 
\item the differential $D(\pi \circ h)$ takes the normal field $\nu|_{Z^-}$ to the field outward  normal to $(\pi\circ h)(Z^-)$ in  $(\pi\circ h)(\d X)$. \hfill\qed  
\end{itemize} 
\end{thm}

Considering a traversing field $v \neq 0$ which is tangent to the fibers of the map $\pi\circ h$ from Theorem \ref{th5.1},  leads instantly to 

\begin{cor}\label{cor5.1} Under the hypotheses and notations from Theorem \ref{th5.1},  there exists a boundary generic traversing field $v$ on $X$ so that: 
\begin{itemize}
\item $\d_1^\pm X(v) = Y^\pm$,   
\item $\d_2^\pm X(v) = Z^\pm$, 
\item $\d_3X(v) = \emptyset$. \hfill\qed
\end{itemize} 
\end{cor}

Thus, at least for smooth domains $X \subset \R^{n+1}$ and for boundary generic traversing fields $v$, which are both $3$-convex and $3$-concave, the patterns for the strata $$\d_2^+X(v) \coprod \d_2^-X(v) \subset \d_1X$$ are indeed very flexible. However, the requirement that both $Z^+ \neq \emptyset$ and $Z^- \neq \emptyset$ puts breaks on any applcation of Corollary \ref{cor5.1} to boundary concave and boundary convex traversing fields on $X$!
\smallskip

{\bf Example 5.1.} Let us illustrate how non-trivial the conclusions of Theorem \ref{th5.1} and Corollary \ref{cor5.1} are. 

Let $X = D^{n+1}$, $n \geq 2$. When $n$ is odd, take any  codimension one submanifold $Z^+ \coprod Z^-\subset S^n$ such that $Z^+ \neq \emptyset$, $Z^- \neq \emptyset$, and $\chi(Z^+) -\chi(Z^-) = 2$. Then $D^{n+1}$ admits a boundary generic traversing field $v$ such that $\d_2^+D^{n+1}(v) = Z^+$ and $\d_2^-D^{n+1}(v) = Z^-$. 

For instance, $D^4$ admits a a boundary concave traversing field $v$ such that $\d_2^-D^4(v) = M^2$, the orientable surface of genus $2$,  and  $\d_2^+D^4(v) = T^2$, the $2$-torus.
\smallskip

When $n$ is even, take any codimension one submanifold $Z^+ \coprod Z^- \subset S^n$ such that $Z^+ \neq \emptyset$, $Z^- \neq \emptyset$, and $\chi(Z^+) = \chi(Z^-)$. Then $D^{n+1}$ admits a boundary generic traversing field $v$ such that $\d_2^+D^{n+1}(v) = Z^+$ and $\d_2^-D^{n+1}(v) = Z^-$. 

For example, for any collection of loops  $Z^+ \coprod Z^- \subset S^2$, $Z^+ \neq \emptyset$, $Z^- \neq \emptyset$, the disk $D^3$ admits a  boundary generic traversing field $v$ such that $\d_2^+D^3(v) = Z^+$ and $\d_2^-D^3(v) = Z^-$.

\hfill\qed


\smallskip

We suspect that an important for our program generalization of Theorem \ref{th5.1} is valid and can be established by the methods as in \cite{E1}, \cite{E2}.

\begin{conj}\label{conj5.1}  Let $X$ be a compact connected smooth manifold of dimension $n+1  \geq 3$, equipped with a traversing vector field $v$. Let $Z^+$ and $Z^-$ be  two disjoint closed and nonempty  $(n-1)$-submanifolds of $\d X$ whose union separates $\d X$ into two $n$-manifolds, $Y^+$ and $Y^-$.  

Then the  topological constraints 
 \begin{eqnarray}\label{eq5.1} 
\chi(Y^+)  & = & \chi(X), \; \;\;\;\;  \text{when} \; n \equiv 0 \; \mod \;2 
\end{eqnarray}
 \begin{eqnarray}\label{eq5.2}
\chi(Z^+) - \chi(Z^-) & =  & 2 \cdot \chi(X), \; \text{when}\;  n \equiv 1 \; \mod \;2
\end{eqnarray}
are necessary and sufficient for the existence of an orientation-preserving diffeomorphism $h: X \to \textup{int}(X)$ with the following properties: 
\begin{itemize}
\item the restriction of $v$ to the image $h(X)$ is boundary generic in the sense of \hfill\break Definition \ref{def2.1} \footnote{and even traversally generic in the sense of Definition 3.2 from \cite{K3}},
\item $\d_1^\pm\big(h(X)\big)(v) = h(Y^\pm)$, 
\item $\d_2^\pm\big(h(X)\big)(v) = h(Z^\pm)$,
\item $\d_3\big(h(X)\big)(v) = \emptyset$.
\end{itemize} 
Moreover, in a given collar $U$ of $\d X$ in $X$, there is a $U$-supported diffeomorphism $h$ as above which is arbitrary close in the $C^0$-topology to the identity map. \hfill\qed
\end{conj}

To prove the necessity of the topological constraints \ref{eq5.1} and \ref{eq5.2} is straightforward.  By the Morse formula \ref{eq5.2} (see also Corollary \ref{cor5.1}), a necessary condition for the existence of a diffeomorphism $h$ with the desired properties, described in the bullets, is  the constraint 
\begin{eqnarray}
\chi(h(X)) - \chi(h(Y^+)) + \chi(h(Z^+)) =  i(v|_{h(X)}) = 0. \nonumber
\end{eqnarray}  

Since $h$ is a homeomorphism, this equation is equivalent to  
\begin{eqnarray}\label{eq5.3}
\chi(X) - \chi(Y^+) + \chi(Z^+)  = 0.
\end{eqnarray}  

If $n \equiv 1\; \mod \;2$, then $$\chi(Y^+) = \frac{1}{2}\chi(\d Y^+)  = \frac{1}{2} (\chi(Z^+) + \chi(Z^-)).$$ Therefore, using formula \ref{eq5.3}, the constraint becomes $2 \chi(X) = \chi(Z^-) - \chi(Z^+)$---formula \ref{eq5.2}. 

When $n \equiv 0 \; \mod \;2$, since $Z^+, Z^-$ are closed odd-dimensional manifolds, $\chi(Z^+) = 0 = \chi(Z^-)$, and formula \ref{eq5.3} reduces to $\chi(X) = \chi(Y^+)$---formula \ref{eq5.1}. 

Therefore the topological constraints  \ref{eq5.1} and \ref{eq5.2} imposed on the ``candidates" $Z^+$, $Z^-$ and $Y^+$ and $Y^-$ are necessary for the existence of the desired diffeomorphism $h$. 
\smallskip

To prove the sufficiency of these conditions may require a clever application of the $h$-principle in the spirit of \cite{E1}, \cite{E2}.
\smallskip







\begin{cor}\label{cor5.2} Assuming the validity of Conjecture \ref{conj5.1}, any compact smooth manifold $X$ with boundary admits a boundary generic traversing field $v$ with  the property $\d_3X(v) = \emptyset$.
\end{cor}

\begin{proof} By Corollary \ref{cor4.1},  $\mathcal V_{\mathsf{trav}}(X) \neq \emptyset$. So we can start with a traversing field $v$ and apply Conjecture \ref{conj5.1} to it to get the pull-back  field $h^\ast(v)$ with the desired properties.
\end{proof}

\begin{conj}\label{conj5.2} Given two vector fields $v_0$ and $v_1$ as in Corollary \ref{cor5.2}, there is a $1$-parameter family of traversing fields $\{v_t\}_{t \in [0, 1]}$ which connects $v_0$ to $v_1$ and such that only for finitely many instances $t \in [0, 1]$, $\d_3X(v_t) \neq \emptyset$. For those  exceptional $t$'s, $\d_4X(v_t) = \emptyset$. \hfill\qed
\end{conj}


\begin{thebibliography}{} 


\bibitem [BP]{BP} Benedetti, R., Petronio, C.,  {\it Branched Standard Spines of 3-manifolds}, Lecture Notes in Mathematics 1653, Springer (1997).





\bibitem[Ca]{Ca} Calabi, E., {\it Characterization of Harmonic $1$-Forms}, Global Analysis, Papers in honor of K. Kodaira, D.C. Spencer and S. Iyanaga, Eds., 1969, 101-117.


\bibitem[D]{D} tomDieck, T., {\it The Burnside Ring of a Compact Lie Group. I},  Math. Ann. 215, 235-250 (1975).

\bibitem [FKL]{FKL} Farber, M., Katz, G., Levine, J., {\it Morse Theory of Harmonic Forms}, Topology 37 (1998), 469-483.




\bibitem [GG]{GG} Golubitsky, M., Guillemin, V., {\it Stable Mappings and Their Singularities}, Graduate Texts in Mathematics 14, Springer-Verlag, New York Heidelberg Berlin, 1973.

\bibitem [GM]{GM} Goresky, M., MacPherson, R., {\it Stratified Morse Theory}, Proceedings of Symposia in Pure Mathematics, Vol. 40 (1983), Part 1, 517-533.

\bibitem [GM1]{GM1} Goresky, M., MacPherson, R., Morse theory for the intersection homology groups,  Analyse et Topologie sur les Espaces Singulieres, Ast\'erisque \#101 (1983), 135-192, Soci\'et\'e Math\'e\-ma\-tique de France.

\bibitem [GM2]{GM2} Goresky, M., MacPherson, R., Stratified Morse Theory, Springer Verlag, New York (1989), Ergebnisse vol. 14. Also translated into Russian and published by MIR Press, Moscow, 1991. 

\bibitem [Go]{Go} Gottlieb, D.H., {\it All the Way with Gauss-Bonnet and  the Sociology of Mathematics}, Math. Monthly, 103 (1996), 457-469.

\bibitem[Gu]{Gu} Guth, L., {\it Minimal number of self-intersections of the boundary of an immersed surface in the plane}, arXiv:0903.3112v1 [math.DG] 18 Mar 2009.

\bibitem[E1]{E1} Eliashberg, Y., {\it Singularities of Folding Type}, Izv. Akad Nauk, 34 (1970), 1110-1126.

\bibitem[E2]{E2} Eliashberg, Y., {\it Surgery of Singularities of Smooth Mappings}, Izv. Akad Nauk, 36 (1972), 1321-1347.








\bibitem [K]{K} Katz, G., {\it Convexity of Morse Stratifications and Spines of 3-Manifolds}, JP Journal of Geometry and Topology, vol. .9, No 1 (2009), 1-119.


\bibitem [K1]{K1} Katz, G., {\it The Burnside Ring-valued Morse Formula for Vector Fields on Manifolds with Boundary}, Journal of Topology and Analysis, vol. 1, No 1, (2009), 13-27.

\bibitem [K2]{K2} Katz, G., {\it Harmonic forms and near-minimal singular foliations}, Comment. Math. Helv., 77 (2002), 39-77.

\bibitem [K3]{K3} Katz, G.,{\it Boundary Generic and Versal Fields: Semi-algebraic Models of Tangency to the Boundary}, preprint (2014).

\bibitem [Kos]{Kos} Kosinski, A., {\it Differential Manifolds}, Academic Press, Boston, San Diego, New York, London, Sydney, Tokyo, Toronto, 1992.



\bibitem [Ma]{Ma} Matveev, S. M., {\it Special spines of piecewise linear manifolds},  Mat. Sb. (N.S.), vol. 92(134), No 2(10) (1973), 282-293. 




\bibitem [MiS]{MiS} Milnor, J.H., Stasheff, J.D., {\it Characteristic Classes}, Annals of Mathematics Studies 76, Princeton University Press, 1974.

\bibitem [Mo]{Mo} Morse, M. {\it Singular points of vector fields under  general boundary conditions}, Amer. J. Math. 51 (1929), 165-178.


\bibitem [P1]{P1} Perelmann, G., {\it The Entropy Formula for Ricci Flow and its Geometric Applications}, arXiv:math.DG/0303109, (2002).

\bibitem [P2]{P2} Perelmann, G.,{\it Ricci flow with Surgery on Three-manifolds}, 
arXiv:math.DG/0303109, (2003).

\bibitem [Po]{Po} Pontriagin, L.S., {\it Smooth Manifolds and their Applications in Homotopy Theory}, Trudy Mat. Inst. Steklov 45 (1955). (= American Math. Soc. Transl. (2) 11, 1955.)



\bibitem[Sm]{Sm} Smale, S., {\it On the structure of manifolds}, Amer. J. Math., 84 (1962), 387-399.

\bibitem [Thom]{Thom} Thom, R., {\it La classification des immersions}, S\'{e}min. Bourbaki 157, 1957-58.




\bibitem [W1]{W1} Whitney, H.,  {\it On regular closed curves in the plane}, Comp. Math. 4 (1937) 276-284.

\end{thebibliography}
\end{document}